\documentclass{article}
\usepackage{fullpage}
\usepackage{textcomp}
\usepackage{amsmath}
\usepackage{amsfonts}
\usepackage{amsthm}
\usepackage{amssymb}
\usepackage[utf8]{inputenc}
\usepackage[square]{natbib}
\usepackage{graphicx}
\usepackage{mathrsfs}
\usepackage{color}
\usepackage[labelfont=bf,textfont=sl,tableposition=top]{caption}
\usepackage{multirow}
\usepackage{setspace}
\usepackage{hyperref}
\usepackage{url}
\usepackage{txfonts}

\theoremstyle{plain}
\newtheorem{thm}{Theorem}
\newtheorem{lem}{Lemma}
\newtheorem{prop} {Proposition}
\newtheorem*{cor}{Corollary}
\theoremstyle{definition}
\newtheorem{exmp}{Example}
\theoremstyle{remark}
\newtheorem*{rem}{Remark}

 \def\footnote#1{{}}

\def\Pr{\mathbb{P}}
\def\E{\mathbb{E}}

\title{Probabilities of concurrent extremes}
\author{Clément Dombry$^\ast$ \and Mathieu Ribatet$^{\dag,\ddag}$ \and Stilian Stoev$^\lozenge$}

\begin{document}
\maketitle
\begin{center}
  $^\ast$ Department of Mathematics, University of Franche-Comté,
  Besançon, FRANCE

  $^\dag$ Department of Mathematics, University of Montpellier,
  Montpellier, FRANCE

  $^\ddag$ Institute of finance and insurance sciences, University of
  Lyon 1, Lyon, FRANCE

  $^\lozenge$ Department of Statistics, University of Michigan, Ann
  Arbor, USA
\end{center}

\begin{abstract}
  The statistical modelling of spatial extremes has recently made
  major advances. Much of its focus so far has been on the modelling
  of the magnitudes of extreme events but little attention has been
  paid on the timing of extremes. To address this gap, this paper
  introduces the notion of \emph{extremal concurrence}. Suppose that
  one measures precipitation at several synoptic stations over
  multiple days. We say that extremes are concurrent if the maximum
  precipitation over time at each station is achieved simultaneously,
  e.g., on a single day. Under general conditions, we show that the
  finite sample concurrence probability converges to an asymptotic
  quantity, deemed \emph{extremal concurrence probability}. Using Palm
  calculus, we establish general expressions for the extremal
  concurrence probability through the max-stable process emerging in
  the limit of the componentwise maxima of the sample. Explicit forms
  of the extremal concurrence probabilities are obtained for various
  max-stable models and several estimators are introduced. In
  particular, we prove that the pairwise extremal concurrence
  probability for max-stable vectors is precisely equal to the
  Kendall's $\tau$. The estimators are evaluated by using simulations
  and applied to study the concurrence patterns of temperature
  extremes in the United States. The results demonstrate that
  concurrence probability can provide a powerful new perspective and
  tools for the analysis of the spatial structure and impact of
  extremes.

  \bigskip
  \noindent \textbf{Keywords:} Max-stable process, Poisson point
    process, Slyvniak formula, Concurrence, Kendall's $\tau$, Temperature.
\end{abstract}

\section{Introduction}
\label{sec:introduction}

While most of the time extreme value analysis focuses on the magnitude
of extreme events, i.e., how large extremes events are, little
interest has been paid to their genesis. This paper tries to fill in
this gap by looking at what we shall call \emph{concurrency of
extremes}, e.g., have two locations been impacted by the same
extreme event or was it a consequence of two different ones? For
example, one could observe daily rainfall at various weather stations
and would like to quantify the risk that the rainfall extremes over a
spatial domain are due to a single extreme event, i.e., a large storm,
affecting the entire area. Although potentially rare, such events have
great socio-economic consequences and their probabilities should be
assessed precisely.

More formally, given a sequence $X_1, \ldots, X_n$ of independent
copies of a stochastic process $X$ defined on a compact set
$\mathcal{X} \subset \mathbb{R}^d$, $d \geq 1$, we say that extremes
are \emph{sample concurrent} at locations
$s_1,\ldots,s_k \in \mathcal{X}$, $k \geq 2$, if
\begin{equation}
  \label{eq:def-concurrence}
  \max_{i = 1, \ldots, n} X_i(s_j) = X_\ell(s_j), \qquad j = 1,
  \ldots, k,
\end{equation}
for some $\ell \in \{1, \ldots, n\}$. Clearly this means that only the
observation $X_\ell$ contributes to the pointwise maxima at
locations $s_1, \ldots, s_k$. It occurs with probability
\begin{equation}
  \label{eq:def-emp-conc-prob}
  p_n(s_1, \ldots, s_k) = \mathbb{P}\left[\text{for some $\ell \in
      \{1,\ldots, n\}\colon \max_{i = 1, \ldots, n} X_i(s_j) =
      X_\ell(s_j)$, $j = 1, \ldots, k$} \right],
\end{equation}
henceforth referred to as \emph{sample concurrence
  probability}.

Provided that $X$ has continuous margins it is not difficult to see
that
\begin{equation*}
  p_n(s_1, \ldots, s_k) = n \mathbb{E} \left[ F\left\{X(s_1), \ldots,
      X(s_k)\right\}^{n-1} \right],
\end{equation*}
where $F$ is the multivariate cumulative distribution of $\{X(s_1),
\ldots, X(s_k)\}$. Interestingly, concurrence of extremes event is invariant
under increasing transformations of the marginals so that the
concurrence probability does not depend on the marginal distributions
of $X$ but only on its dependence structure, i.e., the copula $C$
associated to $F$.

One drawback of the sample concurrence probability
$p_n(s_1, \ldots, s_k)$ is that it varies with the number of
observations $n$.  Surprisingly, however, we prove in
Theorem~\ref{thm:conv} below that, under mild regularity conditions,
this quantity stabilizes to a universal large sample limit
\begin{equation}
  \label{eq:limit}
   p_n(s_1, \ldots, s_k) \longrightarrow p(s_1, \ldots, s_k), \qquad
   n\to\infty.
\end{equation}
Throughout this paper, we will call the limiting probability
$p(s_1, \ldots, s_k)$ the \emph{extremal concurrence
  probability}. This asymptotic quantity is naturally expressed in
terms of a max-stable process $\eta$ emerging in the limit of the
normalized maxima in~\eqref{eq:def-concurrence}, as $n\to\infty$.  A
direct, intuitive, and equivalent definition of the \emph{extremal
  concurrence probability} can be given in terms of the spectral
representation of this max-stable process $\eta$.

Following~\citet{deHaan1984}, \cite{Penrose1992} and
\citet{Schlather2002}, let
\begin{equation}
  \label{eq:spectralCharacterization}
  \eta(s) = \max_{i \geq 1} \zeta_i Y_i(s) \qquad s \in \mathcal{X},
\end{equation}
where $\{\zeta_i\colon i \geq 1\}$ are the points of a Poisson process
on $(0, \infty)$ with intensity measure $\zeta^{-2} \mbox{d$\zeta$}$,
$Y_i$ are independent copies of a non negative stochastic process with
continuous sample paths such that $\mathbb{E}\{Y(s)\} = 1$ for all $s
\in \mathcal{X}$ and $\mathbb{E}\{\sup_{s \in \mathcal{X}} Y(s)\} <
\infty$.  It is often more convenient to
rewrite~\eqref{eq:spectralCharacterization} into
\begin{equation}
  \label{eq:spectralCharacterizationv2}
  \eta(s) = \max_{\varphi \in \Phi} \varphi(s), \qquad s \in \mathcal{X},
\end{equation}
where $\Phi = \{\varphi_i\colon i \geq 1\}$ with
$\varphi_i = \zeta_i Y_i$ is a Poisson point process on
$\mathbb{C}_0$, the space of non negative continuous functions on
$\mathcal{X}$. Within this framework, we now say that extremes are
concurrent at $s_1,\ldots,s_k \in \mathcal{X}$ if
\begin{equation}
  \label{eq:def-concurrence2}
  \eta(s_j) = \varphi_\ell(s_j), \qquad j = 1, \ldots, k,
\end{equation} 
for some $\ell \geq 1$, and similarly to the definition of the
sample concurrence probability~\eqref{eq:def-emp-conc-prob}, the 
extremal concurrence probability is defined by
\begin{equation}
  \label{eq:def-ext-conc-prob}
  p(s_1, \ldots, s_k) = \Pr\left\{\text{for some $\ell\geq 1\colon
      \eta(s_j) = \varphi_\ell(s_j)$, $j = 1, \ldots, k$}
  \right\}.
\end{equation}

When $k=2$ the extremal concurrence probability $p(s_1, s_2)$
coincides with the dependence measure considered
in~\cite{weintraub:1991} to study mixing properties of max-stable
processes. Another well known measure of dependence is the pairwise
extremal coefficient \citep{Schlather2003,Cooley2006}
\begin{equation}
  \label{eq:extremalCoefficient}
  \theta(s_1, s_2) = - \log \Pr\{\eta(s_1) \leq 1, \eta(s_2) \leq 1\},
  \qquad s_1, s_2 \in \mathcal{X}.
\end{equation}
Interestingly, the extremal concurrence probability and the pairwise
extremal coefficient share connections. For instance Proposition 5.1
in~\cite{stoev:2008} implies
\begin{equation}
  \label{eq:comparison}  
  \frac{1}{2} \{2-\theta(s_1,s_2)\} \leq p(s_1,s_2) \leq 2 \{
  2-\theta(s_1,s_2) \},
\end{equation}
and we shall see later that the properties of the extremal concurrence
probability are similar to that of the pairwise extremal coefficient.

The structure of the paper is as follows. In
Section~\ref{sec:concurrent-extremes} we make connections between
sample concurrence probabilities and their extremal counterparts and
derive their properties.  Section~\ref{sec:clos-forms} gives closed
forms for various parametric max-stable models, and
Section~\ref{sec:stat-infer} introduces various estimators for the
sample/extremal concurrence probabilities. The proposed estimators are
then analyzed in a simulation study in
Section~\ref{sec:simul-study-appl} and applied to US continental
temperature extremes in Section~\ref{sec:application}.

\section{Concurrence of extremes} 
\label{sec:concurrent-extremes}

In this section we show that sample concurrence probabilities converge
to extremal concurrence ones under rather mild domain of attraction
conditions. We then provide formulas for the extremal concurrence
probability based on the spectral representation of the associated
max-stable process and establish their basic properties.
 
\subsection{Sample and extremal concurrence}
\label{sec:sample-and-ext-concurrence}

Concurrence of extremes can be defined through the more general notion
of a \emph{hitting scenario}, which reflects precisely how many
different events contribute to the componentwise maximum. Let
$X_1, \ldots, X_n$ be a sequence of independent copies of a stochastic
process $X$ defined on $\mathcal{X}$ and
$s_1, \ldots, s_k \in \mathcal{X}$ be different locations. We suppose
that $X$ has continuous marginals to ensure that
$\{X_1(s_j), \ldots, X_n(s_j)\}$ has no ties almost surely and that
the maximum is uniquely reached. Let
$M_n(s) = \max_{i=1, \ldots, n} X_i(s)$ be the componentwise maximum
and consider the sets $C_i = \{ j\colon M_n(s_j) = X_i(s_j)\}$,
$i = 1, \ldots, n$, that account for the location where the $i$-th
component $X_i$ dominates the rest. Some of these sets may be empty,
but from the above discussion, with probability one the non-empty ones
are disjoint and form a random partition of $\{1, \ldots, k\}$.  This
partition $\pi_n = \{ C_i\colon C_i\neq\emptyset \}$ will be referred
to as the \emph{sample hitting scenario}.

By analogy with extremal concurrence, one can define an \emph{extremal
  hitting scenario} associated to a max-stable process by using the
underlying Poisson point process
\citep{Wang2011,Dombry2013,Dombry2013c}. More precisely, for $\eta$ as
in~\eqref{eq:spectralCharacterizationv2}, the extremal hitting
scenario $\pi$ is defined as the random partition of
$\{1, \ldots, k\}$ such that two indices
$j_1, j_2 \in \{1, \ldots, k\}$ are in the same component of $\pi$ if
and only if
\begin{equation*}
  \underset{i \geq 1}{\arg \max}\ \varphi_i(s_{j_1}) = \underset{i \geq
    1}{\arg \max}\ \varphi_i(s_{j_2}).
\end{equation*}
Whatever type of concurrence is considered, i.e., sample
concurrence~\eqref{eq:def-concurrence} or extremal
concurrence~\eqref{eq:def-concurrence2}, extremes are said concurrent
if and only if $\pi_n = \{1, \ldots, k\}$ or $\pi = \{1, \ldots, k\}$.
The next theorem shows the convergence of the sample hitting scenario
to the extremal one.

\begin{thm}\label{thm:conv}
  Assume that $[h_1\{X(s_1)\}, \ldots, h_k\{X(s_k)\}]$ belongs to the
  maximum domain of attraction of $\{\eta(s_1), \ldots, \eta(s_k)\}$,
  for some strictly increasing deterministic functions $h_i$,
  $i=1, \ldots, k$. Then, the sample hitting scenario $\pi_n$
  converges weakly as $n \to \infty$ to the extremal hitting scenario
  $\pi$ associated to the max-stable process $\eta$.
\end{thm}
\begin{proof}
  Let $X_1, X_2, \ldots$ be a sequence of independent copies of $X$
  observed at some locations
  $s = (s_1, \ldots, s_k) \in \mathcal{X}^k$, $k \geq 2$. For brevity
  we will use componentwise algebra (sum, maximum, etc.) and write
  $X(s) = \{X(s_1), \ldots, X(s_k)\}$. Since the hitting scenario is
  invariant to strictly increasing deterministic transformations of
  the marginals, we can assume without loss of generality that
  $h_i(x) = x$, $i=1, \ldots, k$. By assumptions the componentwise
  maxima
  $M_n(s) = \{\max_{i=1, \ldots, n} X_i(s_1), \ldots, \max_{i=1,
    \ldots, n} X_i(s_k)\}$ converge in distribution
  \begin{equation*}
    \frac{M_n(s) - b_n(s)}{a_n(s)} \longrightarrow \eta(s), \qquad n \to
    \infty,
  \end{equation*}
  where $a_n(s) > 0$ and $b_n(s) \in \mathbb{R}$. It is well known
  that the above convergence is equivalent to the convergence of the
  point process
  \begin{equation*}
    \Phi_n = \left\{ \frac{X_i(s) - b_n(s)}{a_n(s)}\colon i = 1, \ldots,
      n\right\} \subset \mathbb{R}^k.
  \end{equation*}
  to a Poisson point process $\Phi$ on
  ${[0, \infty)}^k \setminus \{0\}$ where the convergence is meant in
  the space of point measures
  $\mathscr{M}_p({[0, \infty]}^k \setminus \{0\})$ equipped with the
  metric of vague convergence \citep{resnick:1987}.
 
  Consider the mapping
  \begin{align*}
    \Pi \colon \mathscr{M}_p({[0, \infty]}^k \setminus \{0\})
    &\longrightarrow \mathscr{P}_k\\
    \Psi &\longmapsto \Pi(\Psi),
  \end{align*}
  where $\mathscr{P}_k$ is the set of all possible partitions of
  $\{1, \ldots, k\}$ and $\Pi(\Psi)$ is the hitting scenario
  associated to the collection of functions
  $\Psi=\{\psi_i\colon i \in I\}$, $I \subseteq \mathbb{N}$. Clearly,
  $\pi_n=\Pi(\Phi_n)$ and $\pi=\Pi(\Phi)$. Since the map $\Pi$ is well
  defined and continuous at each point
  $\Psi \in \mathscr{M}_p({[0, \infty]}^k \setminus \{0\})$ for which
  the maxima are uniquely defined, see
  Appendix~\ref{sec:proof-continuity-theta}, we can apply the
  continuous mapping theorem to show that the weak convergence
  $\Phi_n\to\Phi$ entails the weak convergence $\pi_n\to\pi$.
\end{proof}
\begin{rem}
  A direct consequence of Theorem~\ref{thm:conv} is that, provided the
  stochastic process $X$ is in the domain of attraction of some
  max-stable process $\eta$, the sample concurrence probability
  converges to its extremal counterpart, i.e.,
  \begin{equation*}
    p_n(s_1, \ldots, s_k) = \Pr\left[\pi_n=\{1, \ldots, k\}\right]
    \longrightarrow \Pr\left[\pi= \{1, \ldots, k\} \right]= p(s_1,
    \ldots, s_k), \qquad n \to \infty,
  \end{equation*}
  and proves~\eqref{eq:limit}.
\end{rem}

\subsection{General formulas and properties of extremal concurrence
  probabilities}
\label{sec:ECP-formulae}

The following theorem gives an expression for the extremal concurrence
probability $p(s_1, \ldots, s_k)$.
\begin{thm}\label{thm:thm1}
  We have
  \begin{equation}
    \label{eq:kExtremalConcurrenceProbabilityv1}
    p(s_1, \ldots, s_k) = \mathbb{E}_Y \left( \left[ \mathbb{E}_{\tilde Y}
      \left\{ \max_{j=1, \ldots, k} \frac{\tilde Y(s_j)}{Y(s_j)}
      \right\} \right]^{-1} \right),
  \end{equation}
  where $Y$ and $\tilde Y$ are independent copies of the stochastic
  process appearing in~\eqref{eq:spectralCharacterization}.
\end{thm}
\begin{proof}
  Let $\Lambda$ denotes the intensity mesaure of the
  $\mathbb{C}_0$-valued Poisson point process $\Phi$
  in~\eqref{eq:spectralCharacterizationv2} given by
  \begin{equation*}
    \Lambda(A)=\int_0^\infty \Pr(\zeta Y \in A)\zeta^{-2}
    \mbox{d$\zeta$},
  \end{equation*}
  for all Borel set $A \subset \mathbb{C}_0$. We have
  \begin{align}
    p(s_1, \ldots, s_k) &= \Pr\left\{\exists \varphi \in \Phi\colon
      \varphi(s_1) = \eta(s_1),
      \ldots, \varphi(s_k) = \eta(s_k) \right\} \nonumber\\
    &= \int_{\mathbb{C}_0} \Pr\left\{f(s_j) > \tilde \eta(s_j), j = 1,
      \ldots, k \right\}\Lambda(\mbox{d$f$})\label{eq:thm1.1}\\
    &= \mathbb{E}_Y \left[ \int_0^\infty \Pr\left\{\tilde \eta(s_j) <
        \zeta Y(s_j), j = 1, \ldots, k \right\} \zeta^{-2}
      \mbox{d$\zeta$} \right] \nonumber\\
    &= \mathbb{E}_Y \left( \int_0^\infty \exp \left[-
        \mathbb{E}_{\tilde Y} \left\{ \max_{j=1, \ldots, k}
          \frac{\tilde Y(s_j)}{\zeta Y(s_j)} \right\} \right]
      \zeta^{-2} \mbox{d$\zeta$} \right) \nonumber \\
    &= \mathbb{E}_Y \left( \left[ \mathbb{E}_{\tilde Y}
        \left\{\max_{j=1, \ldots, k} \frac{\tilde Y(s_j)}{Y(s_j)}
        \right\} \right]^{-1} \right)\nonumber
  \end{align}
  where $\tilde \eta$ and $\tilde Y$ are independent copies of $\eta$
  and $Y$ respectively. Note that the second equality uses Slyvniak's
  formula, the fourth one the cumulative distribution of the
  max-stable process $\tilde\eta$ and the last one the expectation of
  an inverse exponential random variable.
\end{proof}

\begin{rem}
  By the seminal paper of~\cite{deHaan1984} (see
  also~\cite{stoev:taqqu:2005,kabluchko:2009}), any continuous in
  probability max-stable process can be represented as
  \begin{equation}
    \label{e:dehaan}
    \{\eta(s)\colon s \in \mathcal{X}\} \stackrel{\rm d}{=} \left\{
      \max_{i \geq 1} \zeta_i f_s(u_i)\colon s \in \mathcal{X}
    \right\},
  \end{equation}
  where $\{f_s\colon s \in \mathcal{X}\}$ is a collection of
  non-negative integrable functions on the space
  $(U, \mathcal{U}, \nu)$. Here $\{(\zeta_i,u_i)\colon i \geq 1\}$ is
  a Poisson point process on $(0,\infty) \times U$ with intensity
  $\zeta^{-2} \mbox{d$\zeta$} \nu(\mbox{d$u$})$.
  
  The functions $\{f_s\colon s \in \mathcal{X}\}$ are known as
  \emph{spectral functions} of $\eta$ and~\eqref{e:dehaan} as de
  Haan's spectral representation. When $\nu$ is a probability measure,
  one can view $Y(s)=f_s$ as random variables on the probability space
  $(U,{\cal U},\nu)$ and then~\eqref{e:dehaan}
  becomes~\eqref{eq:spectralCharacterization}. Conversely, any
  representation~\eqref{e:dehaan} can be cast in the
  form~\eqref{eq:spectralCharacterization} with a change of variables.
  Depending on the context one representation may be more convenient
  than the other. In terms of~\eqref{e:dehaan}, the concurrence
  probability formula in~\eqref{eq:kExtremalConcurrenceProbabilityv1}
  becomes
  \begin{equation}
    \label{eq:kExtremalConcurrenceProbability-deHaan-form}
    p(s_1, \ldots, s_k) = \int_U  \left[ \int_U \left\{ \max_{j=1,
          \ldots, k} \frac{f_{s_j}(\tilde u)}{f_{s_j}(u)} \right\}
      \nu(\mbox{d$\tilde u$}) \right]^{-1} \nu(\mbox{d$u$}),
  \end{equation}
  and the proof is essentially the same.
\end{rem}

One could expect from Definition~\ref{eq:def-concurrence2} and
Theorem~\ref{thm:thm1} that the extremal concurrence probability
depends on the distribution of the spectral process $Y$
in~\eqref{eq:spectralCharacterization} or the choice of spectral
functions in~\eqref{e:dehaan}. These representations are not unique,
but we will see in the theorem below that the extremal concurrence
probability $p(s_1, \ldots, s_k)$ depends only on the distribution of
the max-stable process $\eta$ and not on the choice of the specific
spectral representation.

\begin{thm}\label{thm:thm2}
  For $s_1, \ldots, s_k \in \mathcal{X}$, $k \geq 2$, we have
  \begin{equation}
    \label{eq:kExtremalConcurrenceProbability}
    p(s_1, \ldots, s_k) =  \sum_{r=1}^k (-1)^r \sum_{%
        \begin{smallmatrix}
          J \subseteq \{1, \ldots, k\}\\
          |J| = r
        \end{smallmatrix}
      }
       \mathbb{E}_{\tilde \eta} \left[ \log \Pr_\eta \left\{
          \eta(s_j) \leq \tilde \eta(s_j), j \in J \right\} \right],
  \end{equation}
  where $\tilde \eta$ is an independent copy of $\eta$.
  In particular when $k = 2$,
  \begin{equation}
    \label{eq:bivariateExtremalConcurrenceProbability}
    p(s_1, s_2) = 2 + \mathbb{E}_{\tilde \eta} \left[ \log \Pr_{\eta}
      \left\{ \eta(s_j) \leq \tilde \eta(s_j), j = 1, 2 \right\}
    \right].
  \end{equation}
\end{thm}
\begin{proof}
  Starting from~\eqref{eq:thm1.1} and applying the inclusion-exclusion
  formula, we have
  \begin{align*}
    p(s_1, \ldots, s_k) &= \mathbb{E}_{\tilde \eta} \left[
      \Lambda\left( \left\{\text{$f(s_j) > \tilde \eta(s_j)$ for some
            $j \in \{1, \ldots, k\}$} \right\} \right) \right]\\
    &= \mathbb{E}_{\tilde \eta} \left[ \sum_{r=1}^k \sum_{%
        \begin{smallmatrix}
          J \subseteq \{1, \ldots, k\}\\
          |J| = r
        \end{smallmatrix}
      }
      (-1)^{r+1} \Lambda \left(\left\{f(s_j) > \tilde \eta(s_j), j \in
          J \right\} \right) \right].
  \end{align*}
  Since the cumulative distribution function of $\eta$ is
  $\Pr_\eta \left\{\eta(s_j) \leq \tilde \eta(s_j), j \in J \right\}
  =\exp\left[-\Lambda \left(\left\{f(s_j) > \tilde \eta(s_j), j \in J
      \right\} \right)\right]$, we get
  \begin{equation*}
    p(s_1,\ldots,s_k)= \sum_{r=1}^k (-1)^r \sum_{%
      \begin{smallmatrix}
        J \subseteq \{1, \ldots, k\}\\
        |J| = r
      \end{smallmatrix}
    }
    \mathbb{E}_{\tilde \eta} \left[ \log \Pr_\eta \left\{
        \eta(s_j) \leq \tilde \eta(s_j), j \in J \right\} \right].
  \end{equation*}
  The simplification when $k = 2$ is straightforward because when
  $|J| = 1$ we have
  \begin{equation*}
    \mathbb{E}_{\tilde \eta} \left[ \log \Pr_\eta \left\{\eta(s_j) \leq \tilde
        \eta(s_j), j \in J \right\} \right] =  \mathbb{E}_{\tilde \eta}
    \left\{ - \tilde \eta(s_j)^{-1} \right\} = - 1,    
  \end{equation*}
  as $\tilde \eta(s_j)^{-1}$ is a standard exponential random
  variable.
\end{proof}

In the remaining part of this section, we investigate some properties
of the extremal concurrence probabilities. Surprisingly, although the
two notions are different, we encounter strong similarities with the
extremal coefficient~\eqref{eq:extremalCoefficient}. We recall that
the extremal coefficient $\theta(s_1,s_2)$ takes values in $[1,2]$,
the lower and upper bounds correspond to perfect dependence and
independence respectively. The next proposition states a similar
result for the extremal concurrence probability.
\begin{prop}\label{p:ECP-basics}
  For all $s_1, s_2 \in \mathcal{X}$, we have
  \begin{itemize}
  \item[i)] $p(s_1,s_2)=0$ if and only if $\eta(s_1)$ and $\eta(s_2)$
    are independent;
  \item[ii)] $p(s_1,s_2)=1$ if and only if $\eta(s_1)$ and $\eta(s_2)$
    are almost surely equal.
  \end{itemize}
\end{prop}
The proof uses the following generalization and improvement of the
upper bound in~\eqref{eq:comparison}.

\begin{lem}
  \label{lem:upper-bound}
  For all $s_1, \ldots, s_k \in \mathcal{X}$, $k \geq 2$, we have
  $p(s_1, \ldots, s_k) \leq \mathbb{E} \left\{ \min_{j=1,\ldots,k}
    Y(s_j) \right\}$.
\end{lem}
\begin{proof}
  In the context of Theorem~\ref{thm:thm1}, we have (by conditioning
  on $Y$)
  \begin{equation*}
    \mathbb{E}_{\tilde Y}  \left\{ \max_{j=1, \ldots, k} \frac{\tilde Y(s_j)}{Y(s_j)}
    \right\}  \geq \max_{j=1,\ldots,k} Y(s_j)^{-1}
    \mathbb{E}_{\tilde Y} \left\{\tilde Y(s_j) \right\}= \left\{
      \min_{j=1,\ldots,k} Y(s_i) \right\}^{-1},
  \end{equation*}
  since $\E_{\tilde Y} \{\tilde Y(s_j)\} =1$. This, in view
  of~\eqref{eq:kExtremalConcurrenceProbabilityv1}
  implies the desired result.
\end{proof}

\begin{proof}[Proof of Proposition~\ref{p:ECP-basics}]
  Equation~\eqref{eq:comparison} implies that $p(s_1,s_2)=0$ if and
  only if $\theta(s_1,s_2)=2$ which is equivalent to the independence
  of $\eta(s_1)$ and $\eta(s_2)$. When $p(s_1,s_2)=1$,
  Lemma~\ref{lem:upper-bound} entails $Y(s_1) = Y(s_2)$ almost surely
  so that $\eta(s_1)=\eta(s_2)$ almost surely. It is easy to prove the
  converse implication: if $\eta(s_1)$ and $\eta(s_2)$ are almost
  surely equal, the same holds for $Y(s_1)$ and $Y(s_2)$ so that
  $p(s_1,s_2)=1$.
\end{proof}

Interestingly $p(s_1, \ldots, s_k)$ can be expressed via the extremal
coefficients of another max-stable process.
\begin{prop}
\label{prop:extremalConcurrenceProbabilitiesToExtremalCoefficient}
  Let $\eta$ be a simple max-stable process as defined
  in~\eqref{eq:spectralCharacterization} and
  $\tilde \eta, \tilde \eta_1, \tilde \eta_2, \ldots$ independent
  copies of it. Consider the simple max-stable process
  \begin{equation*}
    \xi(s) = \max_{i \geq 1} \zeta_i \frac{Y_i(s)}{\tilde \eta_i(s)},
    \qquad s \in \mathcal{X},
  \end{equation*}
  then
  \begin{equation*}
    p(s_1, \ldots, s_k) = \sum_{r=1}^k (-1)^{r+1} \sum_{%
      \begin{smallmatrix}
        J \subseteq \{1, \ldots, k\}\\
        |J| = r
      \end{smallmatrix}
    } \theta_\xi(s_j, j \in J),
  \end{equation*}
  where $\theta_\xi(s_j, j \in J) = -\log \Pr \{\xi(s_j) \leq 1, j \in
  J\}$. In particular $p(s_1, s_2) = 2 - \theta_\xi(s_1, s_2)$.
\end{prop}
\begin{proof}
  Clearly $\xi$ is a simple max-stable process since both $Y$ and
  $\tilde \eta$ are non negative and $\mathbb{E}\{Y(s) / \tilde
  \eta(s)\} = 1$ for all $s \in \mathcal{X}$. We have
  \begin{equation*}
    \mathbb{E}_{\tilde \eta} \left[ \log \Pr_\eta \left\{\eta(s_j)
        \leq \tilde \eta(s_j), j \in J\right\} \right] = -
    \mathbb{E}_{\tilde \eta} \left[ \mathbb{E}_Y \left\{ \max_{j \in
          J} \frac{Y(s_j)}{\tilde \eta(s_j)} \right\} \right] = \log
    \Pr_\xi \left\{ \xi(s_j) \leq 1, j \in J\right\} = -
    \theta_\xi(s_j, j \in J).
  \end{equation*}
\end{proof}

The next corollary lists some properties of the extremal concurrence
probability function that closely parallel those of the extremal
coefficient function. In view of
Proposition~\ref{prop:extremalConcurrenceProbabilitiesToExtremalCoefficient},
the proof follows as in \citet{Schlather2003} or \citet{Cooley2006}.

\begin{cor}\label{cor:propertiesOfBivariateExtremalConcurrenceProbability}
  Let $p\colon h \mapsto p(o,h)$ be an extremal concurrence
  probability function associated to a stationary max-stable process
  in $\mathcal{X}$ for some arbitrary origin $o \in \mathcal{X}$ and
  $h \in \mathcal{X}$. Then the following assertions hold.
  \begin{itemize}
  \item[i)] The function $h \mapsto p(h)$ is positive semidefinite;
  \item[ii)] The function $h \mapsto p(h)$ is not differentiable at
    the origin unless $p(h) = 1$ for all $h \in \mathcal{X}$;
  \item[iii)] If $d \geq 1$ and if $\eta$ is isotropic, then $h
    \mapsto p(h)$ has at most a jump at the origin and is continuous
    elsewhere;
  \item[iv)] $\{2 - p(h_1 + h_2)\} \leq \{2 - p(h_1)\} \{2 - p(h_2)\}$
    for all $h_1, h_2 \in \mathcal{X}$;
  \item[v)] $\{2 - p(h_1 + h_2)\}^\alpha \leq \{2 - p(h_1)\}^\alpha +
    \{2 - p(h_2)\}^\alpha - 1$ for all $h_1, h_2 \in \mathcal{X}$ and
    $0 \leq \alpha \leq 1$;
  \item[vi)] $\{2 - p(h_1 + h_2)\}^\alpha \geq \{2 - p(h_1)\}^\alpha +
    \{2 - p(h_2)\}^\alpha - 1$ for all $h_1, h_2 \in \mathcal{X}$ and
    $\alpha < 0$.
  \end{itemize}
\end{cor}

We conclude this section with an unexpected result that relates the
bivariate extremal concurrence probability with the well known
Kendall's $\tau$.

\begin{thm}\label{thm:p-tau}
  For any max-stable process $\eta$, we have $p(s_1,s_2) = \tau$ where
  $\tau = \mathbb{E} \left[\mbox{sign}\{\eta(s_1) - \eta_*(s_1)\}
    \mbox{sign}\{\eta(s_2) - \eta_*(s_2)\} \right]$
  is the Kendall's $\tau$ of $\{\eta(s_1), \eta(s_2)\}$ and $\eta_*$
  is an independent copy of $\eta$.
\end{thm}
\begin{proof}
  Let $W = F\{\eta(s_1), \eta(s_2)\}$ where $F$ is the bivariate
  cumulative distribution function of $\{\eta(s_1), \eta(s_2)\}$.
  From~\eqref{eq:bivariateExtremalConcurrenceProbability} we have
  $p(s_1,s_2) = 2 + \mathbb{E} \left(\log W \right)$. But since
  $\{\eta(s_1), \eta(s_2)\}$ is a bivariate max-stable random vector,
  we know that $\Pr(W \leq w) = w - (1 - \tau) w \log w$,
  $0 \leq w \leq 1$ \citep{ghoudi:khoudraji:rivest:1998} and hence,
  after some simple calculations, $p(s_1, s_2) = \tau$.
\end{proof}

\section{Formulas for extremal concurrence probabilities}
\label{sec:clos-forms}

In this section we gather formulas for the extremal concurrence
probabilities for some popular models of max-stable random vectors and
processes. As we will see, it is not always possible to get explicit
formulas, and in such situations, we propose to use Monte-Carlo
methods.

\subsection{Closed forms}
\begin{exmp}[Logistic model]\label{exmp:logistic}
  The concurrence probability for the $k$-variate logistic model,
  i.e., with cumulative distribution
  \begin{equation*}
    F(z_1, \ldots, z_k) = \exp \left\{ - \left(\sum_{j=1}^k
        z_j^{-1/\alpha} \right)^\alpha \right\}, \qquad 0 < \alpha
    \leq 1, \quad z_1, \ldots, z_k > 0,
  \end{equation*}
  is $p(s_1, \ldots, s_k) = \prod_{j=1}^{k-1} (1 - \alpha / j)$.
\end{exmp}

Recall that for this model independence is reached when $\alpha = 1$
while perfect dependence occurs as $\alpha \downarrow 0$ and, as
expected, for such situations we have $p(s_1, \ldots, s_k) = 0$ and
$p(s_1, \ldots, s_k) = 1$ respectively.

\begin{proof}
  It is not difficult to see that the multivariate logistic model
  corresponds to the case where $Y$
  in~\eqref{eq:spectralCharacterization} is a pure noise process with
  margins such that $\Pr\{Y(s) < y\} = \exp[- \{\Gamma(1 - \alpha)
  y\}^{-1/\alpha} ]$ where $\Gamma$ is the Gamma function. Using
  Theorem~\ref{thm:thm1}, we have
  \begin{align*}
    p(s_1, \ldots, s_k) &= \mathbb{E}_Y \left( \left[
        \mathbb{E}_{\tilde Y} \left\{ \max_{j=1, \ldots, k}
          \frac{\tilde Y(s_j)}{Y(s_j)} \right\} \right]^{-1} \right)
    = \mathbb{E}_Y \left( \left[ - \log F\{\tilde Y(s_1), \ldots,
        \tilde Y(s_k)\} \right]^{-1} \right)\\
    &= \mathbb{E}_Y \left[ \left\{\sum_{j=1}^k Y(s_j)^{-1/\alpha}
      \right\}^{-\alpha} \right] = \frac{\Gamma(k - \alpha)}{\Gamma(k)
      \Gamma(1 - \alpha)} = \prod_{j=1}^{k-1} (1 - \alpha / j),
  \end{align*}
  where the fourth equality used the fact that $\sum_{j=1}^k
  Y(s_j)^{-1/\alpha}$ is a Gamma random variable with scale $\Gamma(1
  - \alpha)^{1/\alpha}$ and shape $k$ for which negative moments are
  known.
\end{proof}

\begin{exmp}[Max-linear model]\label{exmp:max-linear}
  Consider the max-linear model $\eta(s) = \max_{m= 1, \ldots, n}
  \varphi_m(s) Z_m$, where $Z_1, \ldots, Z_n$ are independent unit
  Fréchet random variables and some functions $\{s \mapsto
  \varphi_m(s) \in \mathbb{C}_0, m = 1, \ldots, n\}$ such that
  $\sum_{m=1}^{n} \varphi_m(s) = 1$ for all $s \in \mathcal{X}$. We
  have  
  \begin{itemize}
  \item [i)] The concurrence probability equals 
    \begin{equation}
      \label{eq:p-ell} 
      p(s_1, \ldots, s_k) = \sum_{\ell=1}^n p_\ell(s_1, \ldots, s_k),
      \qquad p_\ell(s_1, \ldots, s_k) = \left\{ \sum_{m=1}^n
        \max_{j=1, \ldots, k} \frac{\varphi_m(s_j)}{\varphi_\ell(s_j)}
        \right\}^{-1},
  \end{equation}
  with the convention that $0/0 = 0$, $a/0 = \infty$ if $a>0$, and
  $1/\infty = 0$.
\item [ii)] The probability that component $\ell$ dominates at sites
  $s_1, \ldots, s_k$ is given by the term $p_\ell$
  in~\eqref{eq:p-ell}, i.e.,
    \begin{equation}
      \label{e:p-ell-max-linear}
      p_\ell(s_1, \ldots, s_k) = \Pr\left\{ \eta(s_j) =
        \varphi_\ell(s_j) Z_\ell,\ j = 1, \ldots, k \right\}    
    \end{equation}
    \end{itemize}
\end{exmp}
\begin{proof}
  Part i) is an immediate consequence
  of~\eqref{eq:kExtremalConcurrenceProbability-deHaan-form}. Indeed,
  let $U = \{1,\ldots,n\}$ be equipped with the counting measure
  $\nu\{1\} = \cdots = \nu\{n\} =1$. By taking
  $f_s(m) = \varphi_m(s)$, $m \in U$, we obtain that the max-linear
  model has the representation~\eqref{e:dehaan}.  The integral
  expression of the concurrence probability
  in~\eqref{eq:kExtremalConcurrenceProbability-deHaan-form} then
  becomes a sum of the terms $p_\ell$ in~\eqref{eq:p-ell}.
  
  Part ii) shows an intriguing fact that the concurrence probability
  $p(s_1,\ldots,s_k)$ for the max-linear model is the sum of the
  probabilities that one of the $n$ components dominates the rest.
  Indeed, by the max-stability property and the independence of the
  unit Fréchet random variables $Z_m$'s, we have that the right-hand
  side of~\eqref{e:p-ell-max-linear} equals
  \begin{equation*}
    \Pr \left\{ \max_{m \neq \ell}
      \max_{j=1,\ldots,k}\frac{\varphi_m(s_j)}{\varphi_\ell(s_j)} Z_m \leq
      Z_\ell \right\} = \Pr(a Z_1 \le Z_2),
  \end{equation*}
  where $a = \sum_{m \neq\ell}
  \max_{j=1,\ldots,k}{\varphi_m(s_j)/\varphi_\ell(s_j)}$.
  Equation~\eqref{e:p-ell-max-linear} follows from the fact that $\Pr(
  a Z_1 \leq Z_2) = (1+a)^{-1}$, $a \geq 0$.
  \end{proof}

\begin{exmp}[Chentsov random fields]
  Suppose that the process $Y(s) = 1_{A}(s)$, where
  $A\subset \mathbb{R}^d$ is a random set. Then, by analogy with the
  theory of symmetric $\alpha$-stable process
  \citep[Chap. 8]{samorodnitsky:taqqu:1994book}, the max-stable
  process
  \begin{equation*}
    \eta(s) = \max_{i \geq 1} \zeta_i 1_{A_i}(s), \qquad s \in
    \mathcal{X},
  \end{equation*}
  where $Y_i\equiv 1_{A_i}$ are independent copies of $Y\equiv 1_A$,
  will be referred to as a \emph{Chentsov max-stable random field} on
  $\mathcal{X}$.

  For a Chentsov-type max-stable process we have
  \begin{equation}
    \label{e:CP-Chentzov}
    p(s_1, \ldots, s_k) =
    \Pr \left ( \{s_1, \ldots, s_k\} \subset A \mid \{s_1,\ldots,s_k\}
      \cap A \neq \emptyset \right),
  \end{equation}
  or less formally that the extremal concurrence probability is the
  conditional probability that the sites $s_1, \ldots, s_k$ are
  covered by the random set $A$ given that at least one of the sites
  is covered.
\end{exmp}
\begin{proof}
  The result is an immediate consequence of
  Theorem~\ref{thm:thm1}. Since the outer expectation
  in~\eqref{eq:kExtremalConcurrenceProbabilityv1} can be restricted to
  the event $\{ \min_{j=1, \ldots, k} Y(s_j)>0\}$, we have
  \begin{align} \label{e:Chentzov-derivation}
    p(s_1, \ldots, s_k) &= \E_Y \left( \left[ \E_{\tilde Y} \left\{
          \max_{j=1, \ldots, k} \frac{1_{\tilde A}(s_j)}{1_{A}(s_j)}
        \right\} \right]^{-1} 1_{\{\min_{j=1, \ldots, k} 1_{A}(s_j) >
        0\}} \right)\nonumber\\
    &= \E_Y \left( \left[ \E_{\tilde Y} \left\{ \max_{j=1, \ldots, k}
          1_{\tilde A}(s_j) \right\} \right]^{-1} \min_{j=1, \ldots,
        k} 1_{A}(s_j) \right)\\
    &= \frac{\E \left\{ \min_{j=1, \ldots, k} 1_{A}(s_j) \right\}
    }{\mathbb{E} \left\{ \max_{j=1, \ldots, k} 1_{A}(s_j) \right\}}, \nonumber
  \end{align}
  where in the second relation we used the fact that $1_{\{\min_{j=1,
      \ldots, k} 1_A(s_j)\}} = \min_{j=1, \ldots, k} 1_A(s_j)$ and
  that the event $\min_{i=1,\ldots,k} 1_{A}(s_i)>0$ is equivalent to
  $\{1_A(s_1) = 1, \ldots, 1_A(s_k) = 1\}$. This
  proves~\eqref{e:CP-Chentzov}.
\end{proof}

\begin{exmp}[Extremal processes]
  The max-stable process $\{\eta(s)\colon s \in [0, 1]\}$ is an
  extremal process if it has stationary and independent
  max-increments, i.e.,
  \begin{equation*}
    \{\eta(s_1), \ldots, \eta(s_k)\} \stackrel{\rm d}{=} \left[ s_1
      Z_1, \max\{s_1 Z_1, (s_2 - s_1)Z_2\}, \ldots, \max\{s_1Z_1,
      \ldots, (s_k - s_{k-1})Z_k\} \right],
  \end{equation*}
  where $0 < s_1 < \cdots < s_k$ and $Z_1, \ldots, Z_k$ are
  independent unit Fréchet random variables. It can be shown that
  \begin{equation*}
    \eta(s) \stackrel{\rm d}{=} \max_{i \geq 1} \zeta_i 1_{[U_i,1]}(s),
    \qquad s \in [0,1],
  \end{equation*}
  where $U_i$'s are independent $U(0,1)$ random variables. Using our
  previous result on Chentsov random fields, we have for all
  $0< s_1 < \cdots <s_k \leq 1$
  \begin{equation*}
    p(s_1, \ldots, s_k) = \frac{\Pr\left( \{ s_1,\ldots, s_k\} \subset
        [U,1]\right)}{\Pr\left( \{ s_1,\ldots, s_k\} \cap [U,1] \neq
        \emptyset \right)} = \frac{\Pr(U \leq s_1)}{\Pr(U \leq s_k)} =
    \frac{s_1}{s_k},
  \end{equation*}
  where $U \sim U(0,1)$. This result is not surprising since for this
  simple case, extremes are concurrent at locations
  $0 < s_1< \cdots< s_k<1$ if $\eta(s)$ has no jumps in the interval
  $[s_1,s_k]$. Hence using the independence and stationarity of the
  max-increments, the probability of the latter event is
  $\Pr \{ s_1 Z_1 \geq (s_k - s_1) Z_2 \} = s_1/s_k$, where $Z_1$ and
  $Z_2$ are two independent unit Fréchet variables.
\end{exmp}

\begin{exmp}[Indicator moving maxima]
  In the context of~\eqref{e:dehaan}, if $f_s(u) = 1_{A_s}(u)$, for
  some sequence of measurable \emph{deterministic sets} $A_s$, by
  using~\eqref{eq:kExtremalConcurrenceProbability-deHaan-form}, we
  obtain as in~\eqref{e:Chentzov-derivation} that
  \begin{equation}
    \label{e:ECP-deterministic-indicators}
    p(s_1,\ldots,s_k) = \frac{\nu(\cap_{j=1,\ldots,k} A_{s_j})}{\nu
      (\cup_{j=1,\ldots,k} A_{s_j})}.
  \end{equation}
  In the simple case $f_s(u) = 1_A(u-s)$, i.e.\ $A_{s} = s+A$ with
  some deterministic set $A$, where $\nu$ is the Lebesgue measure on
  $\mathbb R^d$,~\eqref{e:ECP-deterministic-indicators}~implies
  \begin{equation*}
    p(s,s+h) = p(h)= \frac{|A \cap (h+A)|}{|A \cup (h+A)|} =
    \frac{c_A(h)}{2 |A| - c_A(h)},
  \end{equation*}
  where $|A|$ denotes the $d$-dimensional volume of $A$ and
  $c_A(h) = |A \cap (h+A)|$. The latter function and hence the
  extremal concurrence probability function $p(h)$ can then be
  obtained in closed form for many different sets.  For example, in
  the case $\eta$ is isotropic, i.e.,
  $A = \{ s\in\mathbb R^d\colon \|x\| \leq r\}$ is the centered ball
  of radius $r>0$ in Euclidean space, using the formula for the volume
  of the cap, we obtain
  \begin{equation*}
    c_A(\|h\|) = C_d r^d B_{(d+1)/2,1/2}\left\{\frac{\|h\|
        (2r-\|h\|)}{2r^2} \right\}, \qquad C_d =
    \frac{\pi^{d/2}}{\Gamma(1+d/2)},
  \end{equation*}
  where
  $B_{a,b}(x) = B(a,b)^{-1} \int_0^x u^{a-1}(1-u)^{b-1} \mbox{d$u$}$
  is the cumulative distribution function of a $\mbox{Beta}(a,b)$
  random variable.
\end{exmp}

\subsection{Monte-Carlo methods}\label{sec:Monte-Carlo}

It may happen that for some parametric max-stable models, explicit
forms for extremal concurrence probabilities are not available but
hopefully it is often possible to use Monte-Carlo methods to
approximate the theoretical extremal concurrence probabilities with
arbitrary precision. A naive strategy would consist in
using~\eqref{eq:kExtremalConcurrenceProbabilityv1} to devise a
Monte-Carlo estimator, but it is wiser to take advantage of the closed
forms of max-stable processes cumulative distributions, i.e.,
\begin{equation*}
  \Pr \left\{  \eta(s_j) \leq z_j, j = 1, \ldots, k
  \right\}=\exp\{-V_{s_1,\ldots,s_k}(z_1,\ldots,z_k)\},\qquad z_1,\ldots,z_k>0,
\end{equation*}
where $V_{s_1,\ldots,s_k}$ is an homogeneous function of order $-1$.
Rewriting~\eqref{eq:kExtremalConcurrenceProbabilityv1}, we found
\begin{align}
  p(s_1, \ldots, s_k) &= \mathbb{E}_Y \left( \left[ - \log \Pr_{\tilde
                        \eta} \left\{ \tilde \eta(s_j) \leq Y(s_j), j
                        = 1, \ldots, k 
                        \right\} \right]^{-1} \right)\nonumber\\
                      &= \mathbb{E}_Y \left( \left\{
                        V_{s_1,\ldots,s_k}[Y(s_1),\ldots,Y(s_k)]\right\}^{-1}
                        \right)
                        \label{eq:kExtremalConcurrenceProbabilityv3}
\end{align}
which can easily be estimated by sampling independent copies of $Y$
and computing the sample mean. We can often make use of antithetic
variables to get more precise estimates. Note that specific choice of
the spectral process $Y$ can lead to better strategies as we will
illustrate in the following examples.

\begin{exmp}[Brown--Resnick model]
\label{exmp:brownResnick}
  Let $\eta$ be a Brown--Resnick stationary random field on
  $\mathcal{X}$ driven by a Gaussian process
  \citep{kabluchko:schlather:dehaan:2009}. That is, the processes
  $Y_i$ in~\eqref{eq:spectralCharacterization} are equal in
  distribution to
  \begin{equation}
    \label{e:BR-def}
    Y(s) = \exp\{ W(s) - \gamma(s) \},\qquad s \in \mathcal{X},
  \end{equation}
  where $W$ is a zero mean Gaussian random field with stationary
  increments and semi-variogram $\gamma$, i.e.\ $2\gamma(h) = \E\{
  W(h)^2\} = \E[ \{W(s+h) - W(s)\}^2]$, $s, h \in \mathcal{X}$.
  
  For this model, the bivariate extremal concurrence probability
  function is given by 
  \begin{eqnarray}
    \label{eq:bivPForBrownResnick}
    p(o,h) = \mathbb{E} \left( \left[ \Phi(Z) + \exp\left\{ \gamma(h)
          - \sqrt{2 \gamma(h)} Z \right\} \Phi\left\{\sqrt{2 \gamma(h)}
          - Z \right\} \right]^{-1} \right),
  \end{eqnarray}
  where $Z \sim N(0,1)$ has the standard normal distribution with
  cumulative distribution function $\Phi$.  As expected $p(o) = 1$ and
  $p(h)\to 0$ as $\|h\| \to \infty$ provided that the semi-variogram
  is unbounded, i.e., $\gamma(h) \to \infty$ as $\|h\| \to \infty$.
\end{exmp}

\begin{proof}
  Without loss of generality, for Brown--Resnick processes we can
  assume that in~\eqref{eq:spectralCharacterization} we have $Y(o)=1$
  almost surely.  The bivariate cumulative function is given by
  $\Pr\{\eta(0) \leq z_1,\eta(h) \leq z_2 \} = \exp\{-V_h(z_1,z_2)\},
  \qquad z_1,z_2>0$ with
  \begin{equation*}
    V_h(z_1,z_2) = \frac{1}{z_1} \Phi\left(\sqrt{\gamma(h)/2} +
      \frac{1}{\sqrt{2\gamma(h)}}\log\frac{z_2}{z_1}\right) +
    \frac{1}{z_2}\Phi\left(\sqrt{\gamma(h)/2}+\frac{1}{\sqrt{2\gamma(h)}}\log\frac{z_1}{z_2}\right).
  \end{equation*}
  Equation~\eqref{eq:kExtremalConcurrenceProbabilityv3} together with
  the fact that $\{Y(o),Y(h)\}$ has the same distribution as
  $\{1,e^{\sqrt{2\gamma(h)}\, Z-\gamma(h)}\}$ entails
  \begin{equation*}
    p(o,h) = \mathbb{E} \left( \left\{ V_{h}[Y(o),Y(h)]\right\}^{-1}
    \right) = \mathbb{E} \left\{ \left(
        V_{h}\left[1,\exp\left\{\sqrt{2\gamma(h)}\, Z-\gamma(h) \right\}\right]\right)^{-1} \right\}.
  \end{equation*}
  Equation~\eqref{eq:bivPForBrownResnick} follows after
  straightforward simplifications.
\end{proof}

A popular special case of the the Brown--Resnick family of models is
the moving maximum storm model introduced by~\cite{Smith1990} known
also as the Gaussian extremal process. Consider the spectral
representation~\eqref{e:dehaan}, where $U = \mathbb{R}^d$ and $\nu$ is
the Lebesgue measure. Taking $f_s(u) = \varphi_\Sigma(s-u)$, where
$\varphi_\Sigma$ is the multivariate Normal density with zero mean and
covariance matrix $\Sigma$, we obtain the max-stable process
  \begin{equation}
    \label{eq:smith}
    \eta(s) = \max_{i\ge 1} \zeta_i \varphi_\Sigma(s - u_i), \qquad
    s\in \mathcal{X}.
  \end{equation}
  Then, the following are true:
  \begin{itemize}
  \item[i)] The process $\eta$ belongs to the family of degenerate
    Brown--Resnick models in~\eqref{e:BR-def} with $W(s)= s^\top Z$,
    $s\in {\mathbb R}^d$ and $Z \sim N(0,\Sigma^{-1})$.
 
  \item [ii)] Consequently, the concurrence probability function of
    $\eta$ is given by~\eqref{eq:bivPForBrownResnick} with $\gamma(h)
    = h^\top \Sigma^{-1} h /2 $.
 \end{itemize}

 \begin{proof} Let $\eta$ be a Brown--Resnick process with variogram
   $2 \gamma(h) = h^\top \Sigma^{-1} h$, that is, we can assume
   without loss of generality that in the spectral characterization we
   have $Y(s) = \exp(s^\top \Sigma^{-1} Z - s^\top \Sigma^{-1} s)$
   with $Z \sim N(0, \Sigma)$. Then for all
   $s_1, \ldots, s_k \in \mathcal{X}$ and $z_1, \ldots, z_k >0$ we
   have
  \begin{align*}
    -\log \Pr\{ \eta(s_j)\leq z_j, j=1,\ldots,k\} &= \E_Z \left\{
      \max_{j=1, \ldots, k} z_j^{-1}\exp\left(s_j^\top \Sigma^{-1} Z -
        \frac{1}{2}s_j^\top \Sigma^{-1}
        s_j \right) \right\} \\
    &= (2 \pi)^{-k/2} |\Sigma|^{-1/2} \int_{\mathbb R^d}
    \max_{j=1,\ldots,k} z_j^{-1} \exp\left(-\frac{1}{2} y^\top
      \Sigma^{-1} y + s_j^\top \Sigma^{-1} y - \frac{1}{2}s_j^\top
      \Sigma^{-1} s_j\right) \mbox{d$y$} \\
    &= \int_{\mathbb R^d} \max_{j=1, \ldots, k} z_j^{-1}
    \varphi_{\Sigma}(s_j- y) \mbox{d$y$}.
  \end{align*}
  The last relation equals the negative log cumulative distribution
  function of the moving maxima in~\eqref{eq:smith}.
\end{proof}

\begin{exmp}[Schlather and extremal-$t$ processes]
  Let $\eta$ be an extremal-$t$ process on $\mathcal{X}$, i.e., the
  processes $Y_i$ in~\eqref{eq:spectralCharacterization} are equal in
  distribution to
  \begin{equation*}
    Y(x) = c_\nu \max\{0, W(s)\}^\nu, \quad c_\nu=\sqrt{\pi} 2^{-(\nu-2)/2}
    \Gamma\left(\frac{\nu+1}{2}\right)^{-1}\qquad
    s \in \mathcal{X},
  \end{equation*}
  where $\nu \geq 1$ and $W$ is a stationary standard Gaussian process with
  correlation function $\rho$. The Schlather process is obtained when
  $\nu = 1$.
  
  The corresponding extremal concurrence probability function is
  \begin{equation*}
    p(o,h) = \E\left(\left[T_{\nu+1}(T) +
        \{\rho(h)+\sigma(h)T\}^{-\nu} T_{\nu+1}\left\{
          -\frac{\rho(h)}{\sigma(h)} +
          \frac{1}{\sigma(h)(\rho(h)+\sigma(h)T)}\right\} \right]^{-1}
      1_{\{\rho(h)+\sigma(h)T>0\}}\right)
  \end{equation*}
  where $\sigma(h)=\sqrt{\{1-\rho(h)^2\} / (1+\nu)}$ and $T$ is a
  Student random variable with $\nu+1$ degrees of freedom and
  cumulative distribution function $T_{\nu+1}$.
\end{exmp}
\begin{proof}
  For the notational convenience, we write shortly $\rho=\rho(h)$,
  $\sigma=\sigma(h)$. 
  To obtain the desired result, we use a different spectral
  representation. Comparing the two cumulative distribution functions,
  one can show that
  \begin{equation*}
    \{\eta(o),\eta(h)\} \stackrel{d}{=} \max_{i\geq 1} \zeta_i
    \{\tilde Y_i(o),\tilde Y_i(h)\}
  \end{equation*}
  with $\tilde Y_i$, $i\geq 1$, i.i.d.\ copies of the bivariate random
  vector
  \begin{equation*}
    \{\tilde Y(o), \tilde Y(h)\} = 
    \begin{cases}
      2(1,\max\{0,\rho+\sigma T\}^\nu), &\text{with probability 1/2}\\
      2 (0, c), &\text{with probability 1/2}
    \end{cases}
  \end{equation*}
  with $c=1-\mathbb{E}[\max(0,\rho+\sigma T)^\nu]$ such that
  $\mathbb{E}(\tilde
  Y)=(1,1)$.
  Equation~\eqref{eq:kExtremalConcurrenceProbabilityv3} yields
  \begin{equation*}
    p(o,h) = \E\left(\left[ V_h\{ \tilde Y(o), \tilde Y(h)\}
      \right]^{-1} \right) = \E\left(\left[ V_h\{ 1,(\rho+\sigma
        T)^\nu\} \right]^{-1} 1_{\{\rho+\sigma T>0\}}\right),
  \end{equation*}
  with \citep{Davison2012}
  \begin{equation*}
    V_h(z_1,z_2)=\frac{1}{z_1}T_{\nu+1}\left\{-\frac{\rho}{\sigma} +
      \frac{1}{\sigma} \left(\frac{z_2}{z_1}\right)^{1/\nu} \right\} +
    \frac{1}{z_2} T_{\nu+1}\left\{-\frac{\rho}{\sigma} +
      \frac{1}{\sigma} \left(\frac{z_1}{z_2}\right)^{1/\nu}\right\},
  \end{equation*}
  and straightforward simplifications give the announced result.
\end{proof}

\section{Statistical inference and asymptotic properties}
\label{sec:stat-infer}

\subsection{Sample concurrence probability estimators}

In this section we define a sample concurrence probability estimator
by blocking the data and study its basic properties as well as the
optimal choice of the block-size. We conclude with a methodological
improvement of the estimator based on permutation bootstrap.

Let $X_i = \{X_i(s_j)\colon j = 1, \ldots, k\}$, $i=1, \ldots,n$, be
random vectors in $\mathbb R^k$, $k \geq 2$. Partition the data into
non-overlapping blocks of size $m<n$, and define the \emph{sample
  concurrence probability estimator}
\begin{equation}
  \label{eq:pm}
  \hat p_m \equiv \hat p_m(X_1,\ldots,X_n) = \frac{1}{[n/m]} \sum_{r
    = 1}^{[n/m]} \max_{\ell=1,\ldots,m} 1_{\left\{ \max_{i=1, \ldots,
        m} X_{i+ (r-1)m} \leq X_{\ell+ (r-1)m} \right\}}.
\end{equation}
The max statistic above is simply an indicator of whether or not we
have concurrence in the $r$-th block, i.e., whether one of the vectors
dominates the componentwise maximum of the rest in the $r$-th block
of size $m$.

Assuming that $X_1, \ldots, X_n$ are independent and identically
distributed, the above estimator is the sample mean of $[n/m]$
independent ${\rm Bernoulli}(p_m)$ random variables, where $p_m$ is as
in~\eqref{eq:def-emp-conc-prob} with $n$ replaced by $m$. Therefore,
\begin{equation*}
  \E (\hat p_m) = p_m, \qquad \mbox{Var}(\hat p_m) =
  \frac{p_m(1-p_m)}{[n/m]},
\end{equation*}
i.e., $\hat p_m$ is a unbiased estimator for $p_m$.

As argued in the introduction, a major drawback of the sample
concurrence probability $p_m$ is that it depends on the sample size
$m$ and it is thus more sensible to focus on the limiting
\emph{extremal concurrence probability} $p = p(s_1,\ldots,s_k)$. The
sample concurrence probability estimator $\hat p_m$ is biased for $p$
with mean squared error
\begin{equation}
  \label{e:MSE-pm}
  \mbox{MSE}(\hat p_m) =(p_m-p)^2 + \frac{p_m(1-p_m)}{[n/m]}.
\end{equation}
Although the bias term $p_m-p$ is difficult to estimate in general,
for the max-stable case a precise expression is available.
\begin{prop}\label{prop:bias-estimate}
  Assume that $X$ is a max-stable process, then $p_m$ is
  non-increasing in $m$ and satisfies
\begin{equation*}
  0\leq p_m-p\leq \frac{(1-p)}{m},\qquad  m\geq 1.
\end{equation*}
Furthermore, if $p<1$, then there exists an integer
$r\in \{1,\ldots,k-1\}$ and a positive constant $c_r\in (0,1-p]$, such
that $(p_m - p) \sim c_r/m^r$ as $m \to \infty$.
\end{prop}
The proof expresses $(p_m-p)$ via the distribution of the extremal
hitting scenario and is postponed to
Appendix~\ref{sec:proof-bias-estimate}.

The above result suggests that an asymptotically optimal choice of the
block size can be made to minimize the rate of the mean squared error
in~\eqref{e:MSE-pm}. In view of Proposition~\ref{prop:bias-estimate},
for the general case $0<p<1$,
\begin{equation*}
  \mbox{MSE}(\hat p_m) =\left(\frac{c_r}{m^r}\right)^2 +
  \frac{p(1-p)m}{n} + o(m^{-2r}) + O(m/n)
\end{equation*}
Taking the derivative with respect to $m$ we see that the optimal rate
corresponds to $2r c_r^2m^{-2r-1}\sim p(1-p)n^{-1}$, as $n\to\infty$.
Hence the block size that asymptotically minimizes the mean squared
error is
\begin{equation}
  \label{e:MSE-opt-m(n)}
  m_{\text{MSE}}(n) \sim
  \left\{\frac{2rc_r^{2}n}{p(1-p)}\right\}^{1/(2r+1)}, \qquad n\to\infty.
\end{equation}
For this rate-optimal mean squared error we obtain
$\text{MSE}(\hat p_m) \propto n^{- 2r/(2r+1)}$.

\begin{rem}
  The constants $r$, $c_r$ and $p$ in~\eqref{e:MSE-opt-m(n)} are
  unknown. The precise expressions for $r$ and $c_r$ involve multiple
  concurrence probabilities, i.e., when two or more events contribute
  to the maximum. In principle, pilot estimates of these parameters
  could be obtained and used as plug-ins in~\eqref{e:MSE-opt-m(n)}.
  Since the cases $r\geq 2$ correspond to very specific dependence
  structures, in practice, we recommend using the conservative choice
  $r=1$ and $c_r = 1$.
\end{rem}

The following result establishes the asymptotic behavior of the sample
concurrence probability estimator. The proof is given in
Appendix~\ref{sec:proofs-CLT}.

\begin{thm}\label{thm:CLT-1}
  Suppose that $0<p<1$ and let $m=m(n)$ be such that
  $n/m(n) \to \infty$ and
  $m(n)/n^{1/(2r+1)} \to \lambda \in (0,\infty]$ as $n\to\infty$.
  Then
  \begin{equation*}
  \sqrt{n/m}(\hat p_m - p) \longrightarrow
  N\left\{\frac{c_r}{\lambda^{r+1/2}}, p (1-p)\right\}, \qquad n \to
    \infty
\end{equation*}
where $1 / \infty$ is interpreted as zero.
\end{thm}

\begin{rem}
  In Theorem~\ref{thm:CLT-1}, we encounter a typical tradeoff between
  rate-optimality and bias.  In particular, the mean squared error
  optimal choice of $m$ as in~\eqref{e:MSE-opt-m(n)} corresponding to
  $\lambda = (2rc_r)^{2/(2r+1)}\{p(1-p)\}^{-1/(2r+1)}$ yields the
  limit distribution
  \begin{equation*}
    N\left\{\frac{\sqrt{p(1-p)}}{2r},p(1-p)\right\},
  \end{equation*}
  which has non-zero mean. On the other hand, the rate sub-optimal
  choices where $\lambda=\infty$ yield Normal unbiased limits.
\end{rem}

In the case $p=0$, we have no optimal block size or optimal rate
estimation. Still, the following asymptotic result can be useful.
\begin{thm}
\label{thm:CLT-2}
  Suppose that $p=0$ and let $m=m(n)$ be such that $n/m(n) \to \infty$
  and $m(n)/n^{1/(r+1)} \to \lambda \in (0,\infty]$ as $n\to\infty$.
  If $\lambda <\infty$, then
  \begin{equation*}
    (n/m) \hat p_m  \longrightarrow {\rm Poisson}(c_r/\lambda^{r+1}),
    \qquad n \to \infty.
  \end{equation*}
  Otherwise, if $\lambda = \infty$, then $\Pr(\hat p_m = 0) \to 1$ as
  $n\to\infty$ and hence $a_n \hat p_m \to 0$ in probability as
  $n\to\infty$ for any sequence $a_n>0$.
\end{thm}

\begin{rem}
  When $k = 2$ it is possible to get an unbiased estimator for $p$
  based on a slight modification of $\hat{p}_m$. Indeed for this
  specific case,~\eqref{eq:bias-estimate} implies that
  $p = (m p_m - 1) / (m - 1)$ and hence the estimator
  \begin{equation}
    \label{eq:unbiased-scp}
    \tilde{p}_m = \frac{m \hat{p}_m - 1}{m - 1}
  \end{equation}
  is unbiased for $p$ as $\hat{p}_m$ is unbiased for $p_m$.
\end{rem}

We conclude this subsection with a brief methodological improvement of
the sample concurrence probability estimator $\hat p_m$ based on
permutation bootstrap. The idea is to compute the estimator $\hat p_m$
for several independent random permutations of the sample
$X_1, \ldots, X_n$. Then the average of the resulting estimator would
have a lower variance and the same mean $p_m$.

Formally, this procedure is justified by the following simple
observation based on Rao--Blackwellization.  Consider the
\emph{lexicographic} linear order in $\mathbb R^k$, denoted $\prec$,
and let $X_{(1)} \prec X_{(2)} \prec \cdots \prec X_{(n)}$ be the
sorted sample obtained from $X_1, \ldots, X_n$. The independence of
the $X_i$'s and the continuity of their marginals entails that the
above ordering is strict with probability one.

Let $T\{X_1,\ldots,X_n\}= (X_{(1)},\ldots,X_{(n)})$. It can be shown
that $T$ is a \emph{sufficient statistic} for the parameter
$p_m = p_m(s_1, \ldots, s_k)$ and the Rao--Blackwell theorem implies
the following propostion.

\begin{prop} For $\hat p_m^* = \E (\hat p_m \mid T)$ we have  $\E
  \left(\hat p_m^* \right)= p_m$ and
\begin{equation}
  \label{e:Rao-Blackwell}
  \E\left\{ \left(\hat p_m^* - p_m \right)^2 \right\} \leq \E\left\{
    \left( \hat p_m - p_m\right)^2 \right\}.
\end{equation}
Moreover we have
\begin{equation}
  \label{e:pm-star}
  \hat p_m^* = \frac{1}{n!} \sum_{\sigma \in S_n} \hat p_m \{
  X_{\sigma(1)}, \ldots, X_{\sigma(n)}\},
\end{equation}
and where $S_n$ denotes the set of all permutations of
$\{1,\ldots,n\}$.

An alternative expression for $\hat{p}_m^*$ is
\begin{equation}
  \label{e:pm-star2}
 \hat p_m^*=\frac{1}{{n\choose m}}\sum_{i=1}^n {{d_{i}} \choose {m-1}}
\end{equation}
where $d_i=\sum_{k=1}^n 1_{\{ X_k<X_i\}}$ and
${{d_{i}} \choose {m-1}}=0$ if $d_i<m-1$.
\end{prop}
\begin{proof}
  Let $f(x; p_m)$, $x\in\mathbb R^k$, be the density of
  $X = \{X(s_1),\ldots,X(s_k)\}$ with respect to some dominating
  measure $\lambda$. The fact that for the likelihood, we have
  \begin{equation*}
    L(p_m; X_i, i=1,\ldots,n) = \prod_{i=1}^n f(X_{i}; p_m) = \prod_{i=1}^n f(X_{(i)}; p_m),
  \end{equation*}
  shows that $T = \{X_{(1)},\ldots,X_{(n)}\}$ is a \emph{sufficient
    statistic} for $p_m$. The inequality in~\eqref{e:Rao-Blackwell}
  follows by appealing to the Rao--Blackwell Theorem or simply
  applying the conditional form of Jensen's inequality.

  The independence of the $X_i$'s and the lack of ties (with
  probability one) in the lexicographic order imply
  \begin{equation*}
    \Pr \{ X_{\sigma(i)} = X_{(i)},\ i=1,\ldots,n \} = \frac{1}{n!},
  \end{equation*}
  for all $\sigma \in S_n$. This shows that
  $\hat p_n^* =\E (\hat p_m \mid T)$ is expressed as in~\eqref{e:pm-star}.

  By the definition of the sample concurrence probability estimator,
  we get
  \begin{align*}
    \hat p_m^* &= \frac{1}{n!}\sum_{\sigma\in S_n} \frac{1}{[n/m]}
                 \sum_{r= 1}^{[n/m]} 1_{\left\{\text{sample concurrence
                 occurs within $(X_{\sigma(i+ (r-1)m)}\colon i = 1,
                 \ldots, m)$} \right\}}\\
               &= \frac{1}{{n \choose m}} \sum_{S\in
                 \mathcal{P}_m(n)}1_{\left\{\text{sample concurrence occurs
                 within $(X_i \colon i\in S)$} \right\}}
  \end{align*}
  where $\mathcal{P}_m(n)$ is the collection of all subsets
  $S\subset \{1,\ldots,n\}$ of $m$ elements.  Given a subset
  $S\in\mathcal{P}_m(n)$ and $i_0\in S$, it is easy to see that
  $X_{i_0}$ dominates $(X_{i})_{i\in S}$ if and only if
  $S\setminus \{i_0\}$ is included in the set
  $\{k=1,\ldots,n;X_k<X_{i_0}\}$. We deduce that for a given index
  $i_0$, the number of subset $S\in\mathcal{P}_m(n)$ such that
  $X_{i_0}$ dominates $(X_{i})_{i\in S}$ is equal to
  ${{d_{i_0}}\choose{m-1}}$. Formula~\eqref{e:pm-star2} follows
  easily.
\end{proof}

The above result shows that the estimator $\hat p_m^*$ is superior to
$\hat p_m$ in terms of mean squared error. From a numerical point of
view, formula~\eqref{e:pm-star2} is much more computationally
efficient than~\eqref{e:pm-star}. We shall refer to $\hat p_m^*$ as to
the sample concurrence probability bootstrap estimator. It is
significantly better, in practice, than the simple sample concurrence
probability estimator $\hat p_m$ and therefore in applications we
recommend using only $\hat{p}_m^*$.  As indicated above, in the case
of bivariate concurrence, the bias of $\hat p_m$ can be removed. As in
Relation~\eqref{eq:unbiased-scp}, we obtain the following unbiased
modification of $\hat{p}_m^*$ for the case of pairwise concurrence
\begin{equation}\label{eq:p-tilde_m}
  \tilde p_m^* = \frac{m \hat{p}_m^* - 1}{m - 1}.
\end{equation}
The respective performances of $\hat{p}_m$, $\hat{p}_m^*$ and
$\tilde{p}_m^*$ are analyzed in Section~\ref{sec:simul-study-appl}.

\subsection{Extremal concurrence probability estimators}
\label{subsec:extremal-concurrence-probability-estimators}

Although the sample concurrence probability can be easily estimated,
deriving an estimator for the extremal concurrence probability seems
at first sight difficult since
in~\eqref{eq:kExtremalConcurrenceProbabilityv1} the stochastic
processes $Y$ and $\tilde Y$ are not observable. Recall that for
statistical purposes we often assume that the pointwise block maxima,
e.g.\ pointwise annual maxima, are distributed according to some
max-stable process and thus we observe $\eta$ but not
$Y$. Fortunately, Theorem~\ref{thm:thm2} enables us to estimate
$p(s_1,\ldots,s_k)$ without the need of observing
$Y$.

Based on $\eta_1, \ldots, \eta_n$ independent copies of $\eta$, one
possible estimator for $p(s_1, \ldots, s_k)$ is to consider the sample
counterpart of~\eqref{eq:kExtremalConcurrenceProbability}, i.e.,
\begin{equation}
  \label{eq:logBasedEstimator}
  \hat p(s_1, \ldots, s_k) = \sum_{r=1}^k (-1)^r \sum_{%
    \begin{smallmatrix}
      J \subseteq \{1, \ldots, k\}\\
      |J| = r
    \end{smallmatrix}}
  \frac{1}{n} \sum_{i=1}^n \log \left[
    \frac{1}{n} \sum_{k=1}^n 1_{\{\eta_k(s_j) \leq \eta_i(s_j),\ j \in
      J\}} \right].
\end{equation}

In the innermost summation, we include the index $k=i$ to ensure that
the logarithm is always well defined. Although
estimator~\eqref{eq:logBasedEstimator} seems natural it has
undesirable properties since it is not linear in the data and is thus
likely to show some bias for small sample sizes. Also the study of its
asymptotic properties seems delicate. However, to reduce the bias of
this estimator, it is always possible to use a Jackknife
procedure.

Fortunately, in the bivariate case, it is possible to get an unbiased
estimator based on Theorem~\ref{thm:p-tau}. This theorem suggests the
simple estimator $\hat p \equiv \hat p(s_1, s_2) = \hat \tau$, where
$\hat \tau$ is the Kendall's $\tau$ statistic, i.e.,
\begin{equation}
  \label{eq:kendall-tau}
 \hat p \equiv \hat \tau = \frac{2}{n (n-1)} \sum_{1 \leq i <
    j \leq n} \mbox{sign} \{\eta_i(s_1) - \eta_j(s_1) \} \mbox{sign}
  \{\eta_i(s_2) - \eta_j(s_2) \},
\end{equation}
that is well known to be unbiased and such that \citep[Theorem
4.3]{dengler:2010}
\begin{equation*}
  \sqrt{n}(\hat \tau - \tau) \longrightarrow  N(0,\sigma_\tau^2),
  \qquad \sigma_\tau^2 = 15 \mbox{Var}[ F\{\eta(s_1), \eta(s_2)\} -
  F\{\eta(s_1)\} - F\{\eta(s_2)\} ].
\end{equation*}
Although the asymptotic variance $\sigma^2_\tau$ is hard to evaluate
as it requires knowledge of the dependence structure, in practice it
can be accurately and consistently estimated using Jacknife
\citep{schemper:1987}.

\subsection{Integrated concurrence probabilities and area of concurrence cell}

\label{sec:concurrence-cell}
As we will see in Section~\ref{sec:application}, the above methodology
can be used to provide bivariate concurrence probability maps
$s\mapsto p(s_0,s)$ centered at a given location $s_0$. Such maps show
how fast the dependence in extremes decreases when moving away from
$s_0$. A drawback of this approach is that one may produce one such
map for every choice of an origin $s_0$ and the choice of an origin is
hence quite arbitrary. To bypass this issue, we propose to consider
the integrated concurrence probability
\begin{equation*}
I(s_0)=\int_{s\in\mathcal{X}}p(s_0,s)\mbox{d$s$},\quad s_0\in \mathcal{X}.
\end{equation*}

Intuitively, this quantity measures how fast the dependence in
extremes decreases when moving away from $s_0$. Interestingly, it can
be related to the notion of concurrence cell and its area. Consider
the Poisson process representation of the max-stable random field
$\eta = \{\eta(s)\colon s\in {\cal X}\}$
in~\eqref{eq:spectralCharacterizationv2}.  Recall that we have a
concurrence of extremes at sites $s_0$ and $s$ if
$\eta(s_0) = \phi(s_0)$ and $\eta(s) = \phi(s)$, for the same
$\phi \in \Phi$.  Let $C(s_0)$ denotes the \emph{random} set of all
sites $s$ that are in a concurrence relation with $s_0$.  This set
will be referred to as the \emph{concurrence cell} containing the site
$s_0$.
\begin{prop}\label{prop:ICP} For any site $s_0\in {\cal X}$, we have
  $I(s_0)=\E\{|C(s_0)|\}$ where $|C(s_0)|$ is the $d$-dimensional
  volume of $C(s_0)$.
\end{prop}
\begin{proof} Observe that the concurrence probability satisfies
  $p(s_0,s)=\E\left\{ 1_{C(s_0)}(s) \right\}$ and that the volume of
  the concurrence cell is given by
  \begin{equation*}
    |C(s_0)| = \int_\mathcal{X} 1_{C(s_0)}(s) \mbox{d$s$}.
  \end{equation*}
The result follows by applying the Tonelli--Fubini's theorem.
\end{proof}
We will provide and discuss in Section~\ref{sec:application} some maps
of the integrated concurrence probability $s_0\mapsto I(s_0)$ that
allow to evaluate at each location $s_0$ the dependence in extremes
around $s_0$.  For a detailed study of the properties of the
concurrence cells associated to a max-stable random field and of the
tessellation of the entire domain generated by the concurrence cells,
please refer to the recent work of~\cite{dombry:kabluchko:2014}.

\section{Simulation study}
\label{sec:simul-study-appl}

In this section, we analyze the performance of the pairwise sample
concurrence probability estimators $\hat p_m$, $\hat p_m^*$ and
$\tilde p_m^*$ defined in~\eqref{eq:pm},~\eqref{e:pm-star2}
and~\eqref{eq:p-tilde_m} respectively, and that of their extremal
counterpart $\hat{p}$ in~\eqref{eq:kendall-tau}. Since the latter
estimator relies on the max-stability assumption while the former
three assume that observations belong to the max-domain of attraction,
we need to handle both situations.  The first one is well known and
consists in sampling from max-stable processes using the methodology
of \citet{Schlather2002}. In the second situation, to be able to
control the degree to which the model differs from a max-stable one,
we consider the following partial maxima
\begin{equation}
  \label{eq:truncated-spectral-charac}
  \tilde{\eta}(s) = \frac{1}{n_0} \max_{i=1, \ldots, n_0} U_i^{-1}
  Y_i(s), \qquad s \in \mathcal{X},
\end{equation}
where $Y_i$ are as in~\eqref{eq:spectralCharacterization},
$U_1, \ldots, U_{n_0}$ independent $U(0,1)$ random variables and for
some suitable $n_0 \in \mathbb{N}$. By construction, $\tilde{\eta}$
belongs to the max-domain of attraction of $\eta$
in~\eqref{eq:spectralCharacterization} and in some sense can be viewed
as a \emph{truncation} of the spectral representation
in~\eqref{eq:spectralCharacterization} (see, e.g.,\ the proof of
Proposition 3.1 in~\cite{stoev:taqqu:2005}.)

\begin{figure}
  \centering
  \includegraphics[width=\textwidth]{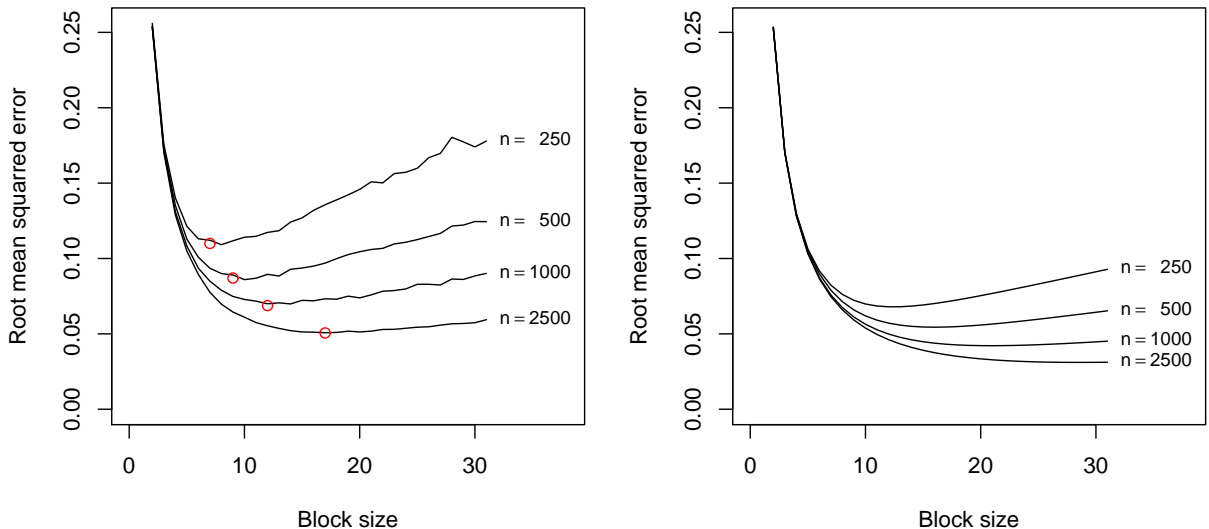}
  \caption{Evolution of the root mean squared error for $\hat{p}_m$
    (left) and $\hat{p}_m^*$ (right) as the block size $m$ and the
    sample size $n$ increase. These estimates were obtained from 2000
    Monte-Carlo samples sampled from a Brown--Resnick model with
    semivariogram $\gamma(h) = h / 1.627$. This semivariogram was
    chosen such that the theoretical extremal concurrence probability
    is $p = 0.5$. The red circles indicate the optimal block sizes as
    defined by~\eqref{e:MSE-opt-m(n)} and their corresponding optimal
    root mean squared error~\eqref{e:MSE-pm}.}
\label{fig:perfSCP}
\end{figure}

We first focus only on the sample concurrence probability estimators,
i.e., $\hat{p}_m$ and $\hat{p}_m^*$, and analyze their performance
with respect to the block size $m$ and the sample size $n$. Based on a
Monte-Carlo simulation, Figure~\ref{fig:perfSCP} shows the evolution
of the root mean squared error as the block size grows. As expected,
both estimators become increasingly more efficient as the sample size
grows and, as seen from~\eqref{e:Rao-Blackwell}, the permutation
estimator $\hat{p}_m^*$ is more efficient than
$\hat{p}_m$---independently of the block size $m$ and the sample size
$n$.  The circles on the plot indicate the asymptotically optimal
block size in~\eqref{e:MSE-opt-m(n)}, which are valid olnly for
max-stable data. As expected the observed optimal block sizes are in
good agreement with the theoretical ones. In practice, however, since
the data are not exactly max-stable, we recommend using slightly
larger values of $m$ so as to ensure that the block-maxima are closer
to a max-stable model but also to take into account that data usually
exhibit serial dependence, e.g., daily
observations.

\begin{figure}
  \centering
  \includegraphics[width=\textwidth]{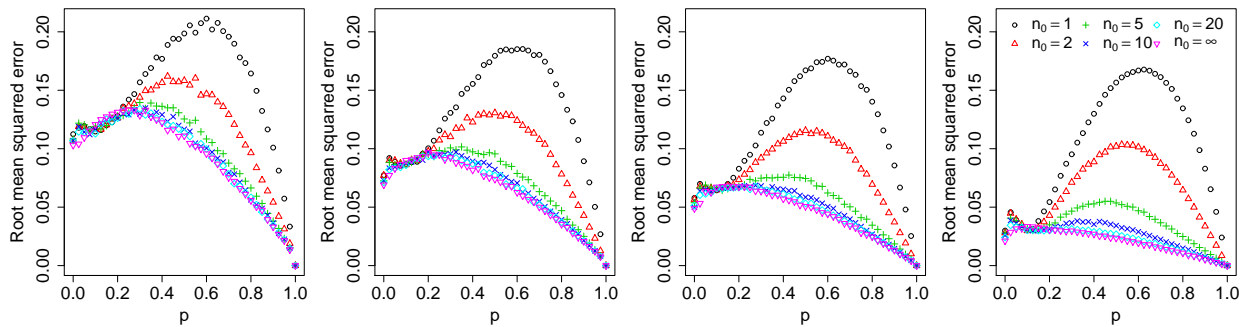}
  \caption{Evolution of the root mean squared error for $\hat{p}$ as
    the theoretical extremal concurrence probability $p$ and the
    number of spectral function $n_0$
    in~\eqref{eq:truncated-spectral-charac} increase. These estimates
    were obtained from 2000 Monte-Carlo samples of size $n$ with,
    from left to right, $n = 25, 50, 100, 500$.}
\label{fig:perfECP}
\end{figure}

We now investigate the performance of the extremal concurrence
estimator $\hat{p}$. Figure~\ref{fig:perfECP} shows the evolution of
the root mean squared error as the number of spectral functions $n_0$
in~\eqref{eq:truncated-spectral-charac} and the theoretical extremal
concurrence probability increase. As expected, as the sample size $n$
grows, the estimator $\hat{p}$ becomes much more efficient. Interestingly, for
small sample sizes, $\hat{p}$ appears to be fairly robust to the lack
of max-stability in the data, i.e., $n_0 < \infty$. This is not true
anymore for larger sample sizes since, as expected, $\hat{p}$ becomes
increasingly more efficient as the number of spectral functions
increases.

Finally, we compare the performance of the sample concurrence 
probability estimators $\hat{p}_m^*$ and $\tilde p_m^*$ with their
extremal concurrence counterpart $\hat{p}$. To compare the two types 
of estimators on a fair basis, we analyze their behaviour when the simulated 
data are either perfectly max-stable or non max-stable, but in the domain of 
attraction of a max-stable distribution.
\begin{figure}
  \centering
  \includegraphics[width=\textwidth]{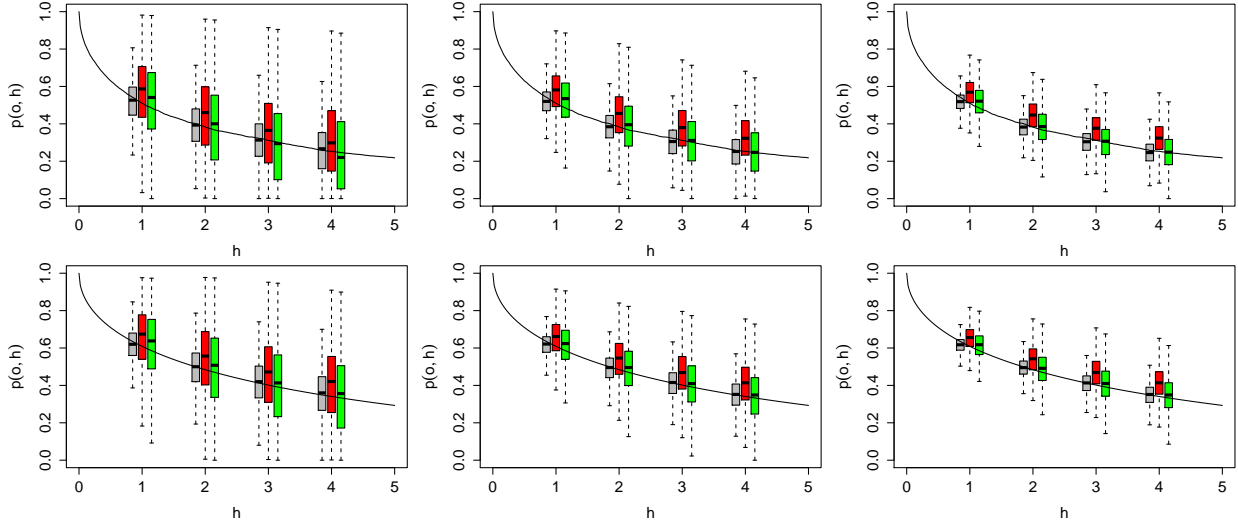}
  \caption{Boxplots of the sample (red / middle), unbiased sample
    (green / right) and extremal (grey / left) concurrence probability
    estimators at distance lags $h=1,2,3,4$. The boxplots were
    obtained from 2000 independent estimates. From left to right: the
    sample size is respectively 25, 50, 100 and 500. The top panel
    corresponds to an extremal-$t$ model with $\nu = 5$, and
    correlation function $\rho(h) = \exp(-h/10)$. The bottom panel
    corresponds to a Brown--Resnick model with semi variogram
    $\gamma(h) = h / 3$. For each panel, the solid line represents the
    corresponding theoretical extremal concurrence probability
    function.}
\label{fig:comparisonECPandSCP}
\end{figure}

Figure~\ref{fig:comparisonECPandSCP} shows boxplots of the sample
$\hat p_m^*$, unbiased sample $\tilde p_m^*$ and extremal concurrence
probability estimators $\hat p$, based on 2000 Monte-Carlo
realizations of both a Brown--Resnick and extremal-$t$ models.  Recall
that we focus here on pairwise concurrence probabilities.  In this
case, the extremal concurrence probability coincides with Kendall's
$\tau$ and therefore, the estimator $\hat p$ in~\eqref{eq:kendall-tau}
is in fact unbiased for the case of max-stable data. This is
confirmed by the results in Figure~\ref{fig:comparisonECPandSCP}.  As
expected, the variability of all estimators decreases as the sample
size grows; the extremal concurrence probability estimator being the
most precise one. Since the simulated data are max-stable, we can see
that the sample concurrence probability estimator is biased even when
the sample size is large while the remaining two estimators are, as
expected, unbiased. Overall the extremal concurrence probability
appears to be the best estimator provided that the data are
max-stable.

\begin{table}
  \centering
  \caption{Performance of the sample $(\hat{p}^*_m)$, unbiased sample
    $(\tilde{p}^*_m)$ and extremal $(\hat{p})$ concurrence probability
    estimators. The table report the sample mean and the standard
    deviation in paren based on 2000 Monte-Carlo replicates. The data
    are either simulated from an extremal-$t$ model with correlation
    function $\rho(h) = \exp(-h / 10)$ and $\nu = 5$ degrees of
    freedom or from its truncated representation with $n_0$ extremal
    functions. Throughout this simulation study the block size is held
    fixed to $m = 10$, independently of the sample size $n$.}
\label{tab:comparisonECPandSCP}
  {\scriptsize
  \begin{tabular}{lccccccccccc}
    \hline
    & \multicolumn{3}{c}{$p = 0.25$} && \multicolumn{3}{c}{$p = 0.50$} && \multicolumn{3}{c}{$p = 0.75$}\\
    \cline{2-4} \cline{6-8} \cline{10-12}
    & $\hat{p}^*_m$ & $\tilde{p}^*_m$ & $\hat{p}$ && $\hat{p}^*_m$ & $\tilde{p}^*_m$ & $\hat{p}$ && $\hat{p}^*_m$ & $\tilde{p}^*_m$ & $\hat{p}$\\
    \hline
    \multicolumn{12}{c}{Sample size $n = 20$}\\
         $n_0 = \phantom{1}1$ & $0.41~(0.24)$ & $0.35~(0.26)$ & $0.47~(0.13)$ && $0.64~(0.22)$ & $0.60~(0.24)$ & $0.71~(0.09)$ && $0.83~(0.15)$ & $0.81~(0.17)$ & $0.87~(0.05)$ \\ 
        $n_0 = 10$ & $0.34~(0.24)$ & $0.27~(0.25)$ & $0.31~(0.14)$ && $0.57~(0.23)$ & $0.52~(0.26)$ & $0.58~(0.12)$ && $0.79~(0.17)$ & $0.77~(0.19)$ & $0.80~(0.07)$ \\ 
        $n_0 = 15$ & $0.33~(0.24)$ & $0.27~(0.25)$ & $0.30~(0.15)$ && $0.56~(0.23)$ & $0.52~(0.26)$ & $0.56~(0.12)$ && $0.78~(0.17)$ & $0.76~(0.19)$ & $0.78~(0.07)$ \\ 
    $n_0 = \infty$ & $0.33~(0.24)$ & $0.27~(0.25)$ & $0.25~(0.15)$ && $0.55~(0.24)$ & $0.50~(0.26)$ & $0.50~(0.13)$ && $0.77~(0.18)$ & $0.75~(0.20)$ & $0.75~(0.08)$ \\ 
    \multicolumn{12}{c}{Sample size $n = 50$}\\
         $n_0 = \phantom{1}1$ & $0.41~(0.13)$ & $0.35~(0.14)$ & $0.47~(0.08)$ && $0.65~(0.10)$ & $0.61~(0.12)$ & $0.71~(0.05)$ && $0.84~(0.07)$ & $0.82~(0.07)$ & $0.87~(0.03)$\\ 
        $n_0 = 10$ & $0.34~(0.13)$ & $0.26~(0.14)$ & $0.31~(0.09)$ && $0.57~(0.12)$ & $0.52~(0.13)$ & $0.57~(0.07)$ && $0.79~(0.08)$ & $0.76~(0.09)$ & $0.80~(0.04)$ \\ 
        $n_0 = 15$ & $0.33~(0.13)$ & $0.25~(0.14)$ & $0.29~(0.09)$ && $0.56~(0.12)$ & $0.51~(0.13)$ & $0.56~(0.07)$ && $0.78~(0.08)$ & $0.76~(0.09)$ & $0.79~(0.04)$ \\ 
    $n_0 = \infty$ & $0.32~(0.13)$ & $0.25~(0.14)$ & $0.24~(0.09)$ && $0.54~(0.12)$ & $0.49~(0.14)$ & $0.50~(0.08)$ && $0.77~(0.09)$ & $0.74~(0.09)$ & $0.75~(0.05)$ \\ 
    \multicolumn{12}{c}{Sample size $n = 100$}\\
         $n_0 = \phantom{1}1$ & $0.41~(0.08)$ & $0.35~(0.09)$ & $0.46~(0.06)$ && $0.65~(0.07)$ & $0.61~(0.07)$ & $0.71~(0.03)$ && $0.83~(0.04)$ & $0.82~(0.04)$ & $0.87~(0.02)$\\ 
        $n_0 = 10$ & $0.34~(0.09)$ & $0.26~(0.10)$ & $0.31~(0.06)$ && $0.57~(0.08)$ & $0.52~(0.09)$ & $0.57~(0.05)$ && $0.78~(0.05)$ & $0.76~(0.05)$ & $0.80~(0.03)$ \\ 
        $n_0 = 15$ & $0.33~(0.09)$ & $0.26~(0.10)$ & $0.29~(0.06)$ && $0.56~(0.08)$ & $0.51~(0.09)$ & $0.55~(0.05)$ && $0.78~(0.05)$ & $0.76~(0.06)$ & $0.78~(0.03)$ \\ 
    $n_0 = \infty$ & $0.33~(0.09)$ & $0.25~(0.10)$ & $0.25~(0.07)$ && $0.55~(0.08)$ & $0.50~(0.09)$ & $0.50~(0.05)$ && $0.78~(0.05)$ & $0.75~(0.06)$ & $0.75~(0.03)$ \\    
    \hline
  \end{tabular}
}
\end{table}

To corroborate this finding, Table~\ref{tab:comparisonECPandSCP}
reports Monte-Carlo sample means and standard deviations of these
estimators as the assumption of max-stability becomes more accurate,
i.e., as the number $n_0$ of spectral functions
in~\eqref{eq:truncated-spectral-charac} grows. As expected, when the
max-stability assumption is most unreasonable, i.e., $n_0 = 1$, all
estimators show a substantial bias with the extremal concurrence
probability estimator $\hat p$ having the largest bias while the
unbiased sample one $\tilde p_m^*$ the lowest. As the assumption of
max-stability becomes increasingly more accurate, the bias of the
unbiased sample concurrence and extremal concurrence estimators
improve.  When this assumption holds exactly (indicated by
$n_0=\infty$), both of these estimators exhibit essentially no bias as
stipulated by the theory and seen in
Figure~\ref{fig:comparisonECPandSCP}. The sample concurrence
probability estimator appears to be biased in all situations---the
bias being less significant as the number of spectral functions is
larger. Interestingly, whatever the estimator considered, the bias and
variance appear to increase as the theoretical extremal concurrence
probability value $p$ becomes smaller. Overall the extremal
concurrence probability estimator $\hat p$ in~\eqref{eq:kendall-tau}
has the lowest variability.

\section{Concurrence of temperature extremes in continental USA}
\label{sec:application}

\begin{figure}
  \centering
  \includegraphics[width=\textwidth]{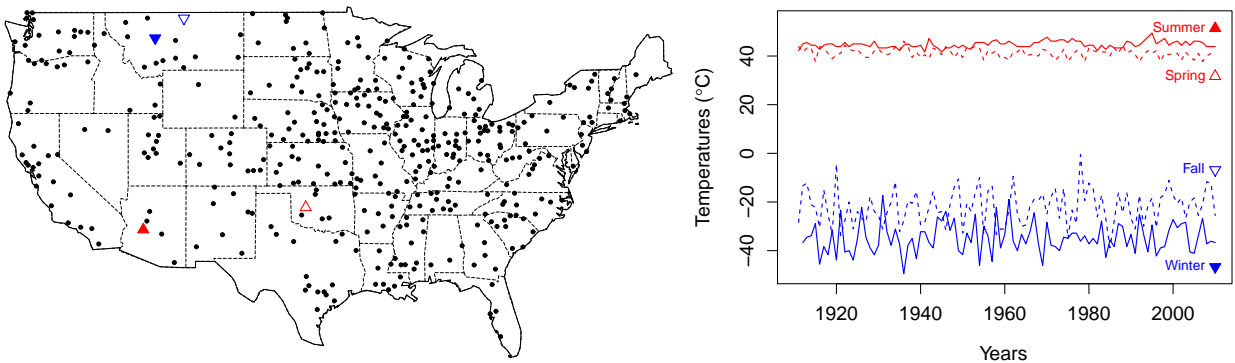}
  \caption{Left: Spatial distribution of the 424 weather stations. The
    triangles indicate the selected stations for the
    analysis---upward: daily maxima, downward: daily minima. Right:
    The seasonal extrema time series of the selected stations.}
\label{fig:study-region}
\end{figure}

In this section, we apply the developed methodology to estimate the
probabilities of concurrence associated with extreme
temperatures---both extreme cold and hot events. The data consists of
daily temperature minima and maxima recorded at 424 weather stations
over the period 1911--2010. The spatial distribution of these stations
is given in Figure~\ref{fig:study-region}. This data set, as a subset
of the United States Historical Climatological Network
\citet{USHCN:424stations}, was chosen as it meets very high data
quality standards and involves fewer than 2.4\% missing values while
spanning the entire territory of continental US\@. It can be freely
downloaded from \href{http://cdiac.ornl.gov}{http://cdiac.ornl.gov/}.

To avoid any seasonal influence on our results we decided to analyze
minima and maxima for each season separately.  We focus on the
concurrence of extreme cold (minima) during the Fall and Winter
seasons---generally color-coded in blue; and extreme hot (maxima)
during the Spring and Summer seasons---generally color-coded in red.
The right panel of Figure~\ref{fig:study-region} shows the times
series of these seasonal extrema for four selected weather
stations. These stations were selected as they recorded the top
seasonal records over the whole spatial and temporal domains. We can
see that all four time series of seasonal extremes (cold in blue and
hot in red) at these stations appear to be stationary without any
clear temporal trend.  This is in contrast with the generally accepted
trend of about $0.2^\circ{\rm C}$ per decade for average temperatures
\citep{ipcc2013}.

\begin{figure}
  \centering
  \includegraphics[width=\textwidth]{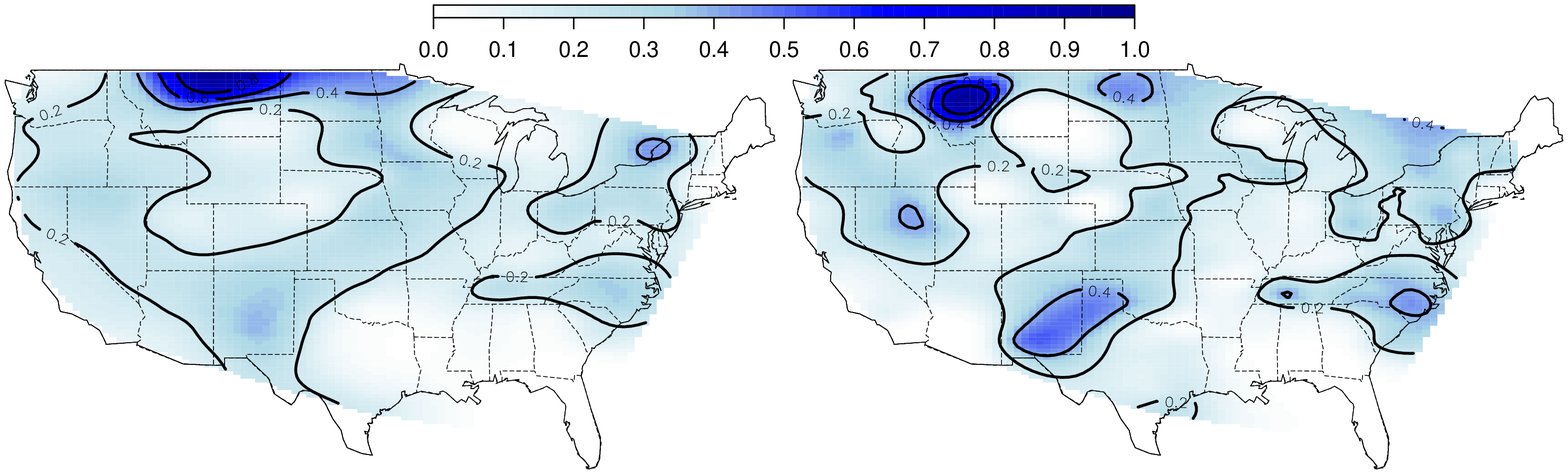}\\\vspace{-1.5em}
  \includegraphics[width=\textwidth]{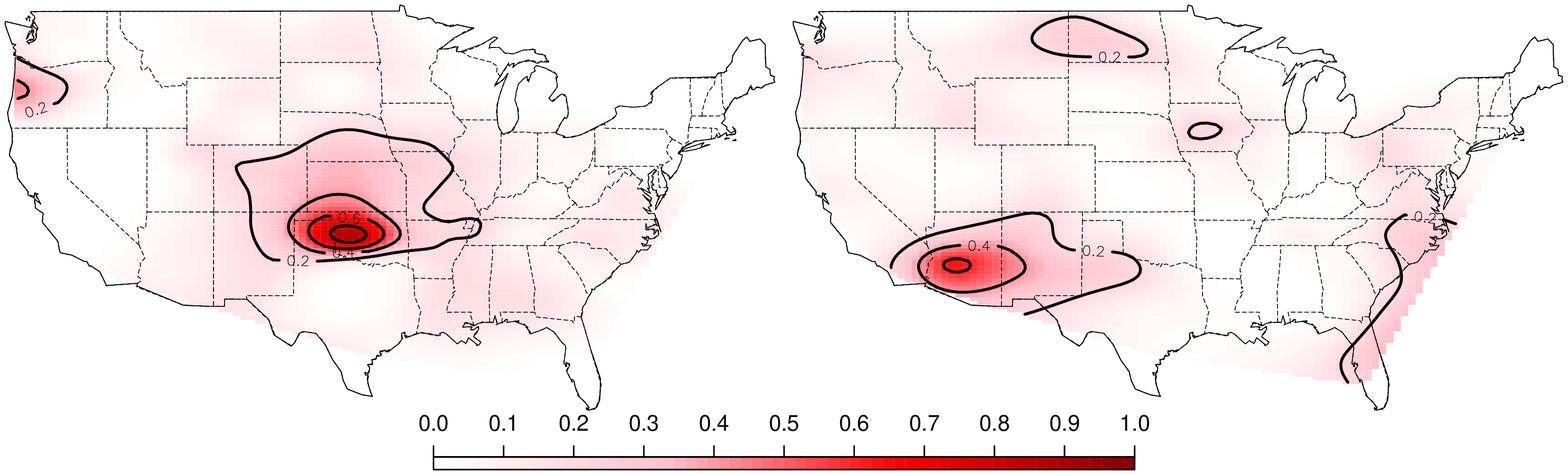}
  \caption{Maps of the extremal concurrence probability for the four
    selected stations. Top left: Fall (September, October November), top right:
    Winter (December, January, February), bottom left: Spring (March, April, May)
    and bottom right: Summer (June, July, August).}
\label{fig:concurrenceProbabilityAcrossSeasons}
\end{figure}

Figure~\ref{fig:concurrenceProbabilityAcrossSeasons} plots the
estimated spatial distribution of the extremal concurrence
probabilities function for each season, relative to the chosen
station.  More precisely, for a given origin location $s_0$, the maps
display estimates of the pairwise concurrence probability $p(s_0,s)$
as a function of $s$.  These maps were obtained by first computing the
estimator~\eqref{eq:kendall-tau} over all $423$ pairs of stations
$(s_0,s)$ and then interpolated using thin plate splines (on logit
scale) provided by the \texttt{R} package \texttt{fields}
\citep{fields}. As expected, the highest concurrence probability
occurs in the neighbourhood of the selected stations independently of
the season. The areal extent of high concurrence probabilities,
however, seem to be larger for minimum temperatures (cold extremes)
than for maximum temperatures (hot extremes).  This finding is
consistent with the physical notion of entropy, i.e., when the ambient
temperature is higher (Spring and Summer seasons), the entropy is
greater and hence involves less spatial dependence than for cooler
temperatures leading to smaller probability of simultaneous extremes.
This difference can be also attributed to the fact that extreme cold
temperatures are often due to high-pressure systems, which tend to
linger longer and cover a larger spatial area than warm fronts giving
rise to concurrence of extreme hot events.

\begin{figure}
  \centering
  \includegraphics[width=\textwidth]{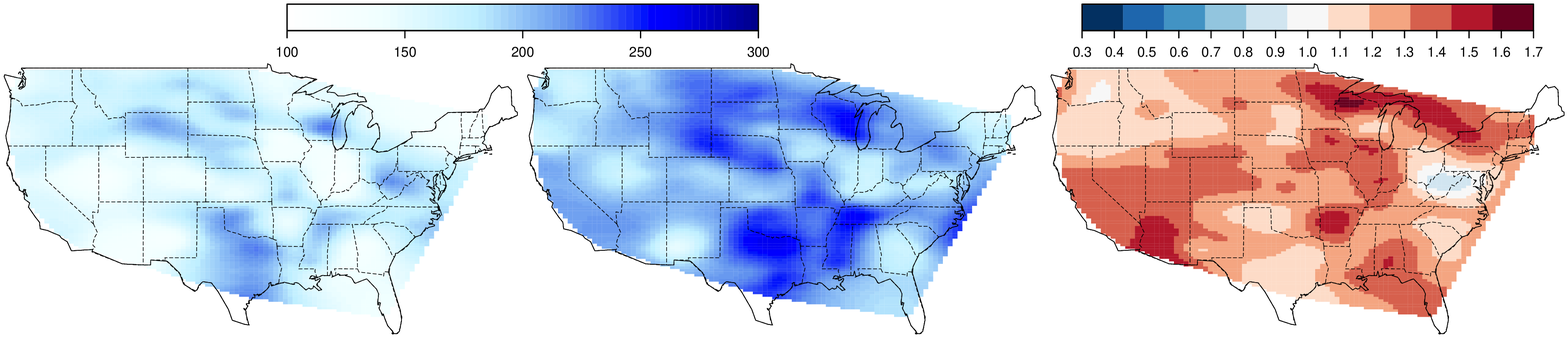}\\ \vspace{-1em}
  \includegraphics[width=\textwidth]{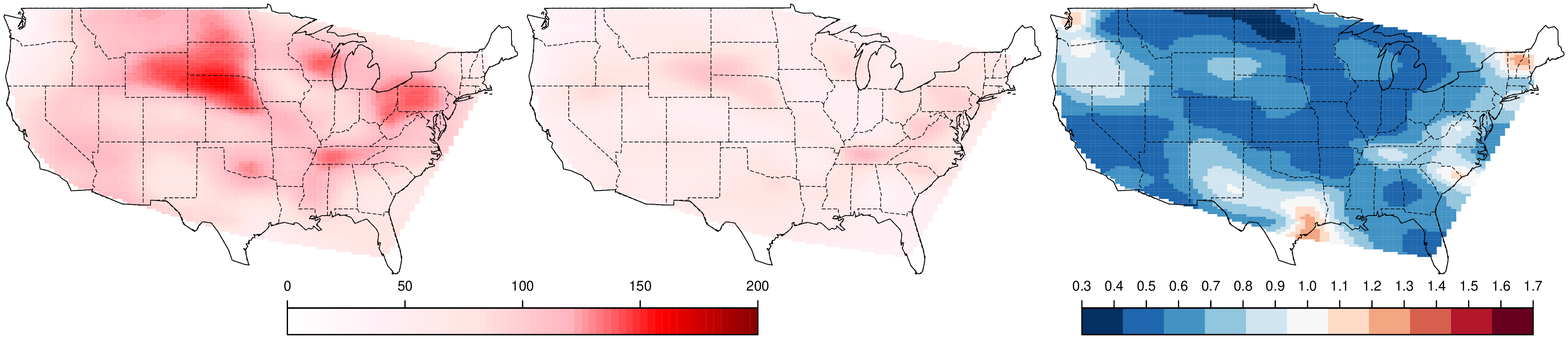}%
  \caption{Estimated spatial distribution of the expected extremal
    concurrence cell areas---in squared degree, i.e., around
    $1000~\mbox{km}^2$. From left to right: 1910--1950, 1951--2010,
    and their ratio (1951--2010 at the numerator). Top: Winter minima,
    bottom: Summer maxima.}
\label{fig:cellAreaDistribution}
\end{figure}

Although Figure~\ref{fig:concurrenceProbabilityAcrossSeasons} displays
interesting patterns, it has the drawback of being dependent on the
choice of the origin, i.e., the selected station. As stated in
Section~\ref{prop:ICP}, it is possible to bypass this hurdle by
considering the area of concurrence
cell. Figure~\ref{fig:cellAreaDistribution} plots the estimated
spatial distribution of the concurrence cell area for the
preindustrial period, i.e., 1910--1975, and the postindustrial one,
i.e., 1976--2010. To emphasize the possible impact of anthropogenic
influences, the ratio of these two cell areas is also reported. We can
see that during the last sixty years the expected cell area for winter
minima have increased of about 30\% over the whole USA while there is
a decrease of about the same amount for summer maxima. These findings
indicates that today's climate shows cold spells that have a larger
impact than in the beginning of the 20th century while hot spells are
more localized. Our results agree with the conclusions drawn by
\citet{ipccExtremes2012} who states that ``there is evidence from
observations gathered since 1950 of change in some extremes''. These
changes in the concurrence patterns of summer extremes can be
attributed to global warming since an increase in entropy generally
leads to more ``mixing" in the system and hence less dependence
leading to smaller areas of concurrence. The changes in concurrence
patterns of extreme cold events, however, are harder to explain. They
may be triggered by structural changes in important climatological
mechanisms such as the Arctic Oscillation.

\begin{figure}
  \centering
  \includegraphics[width=\textwidth]{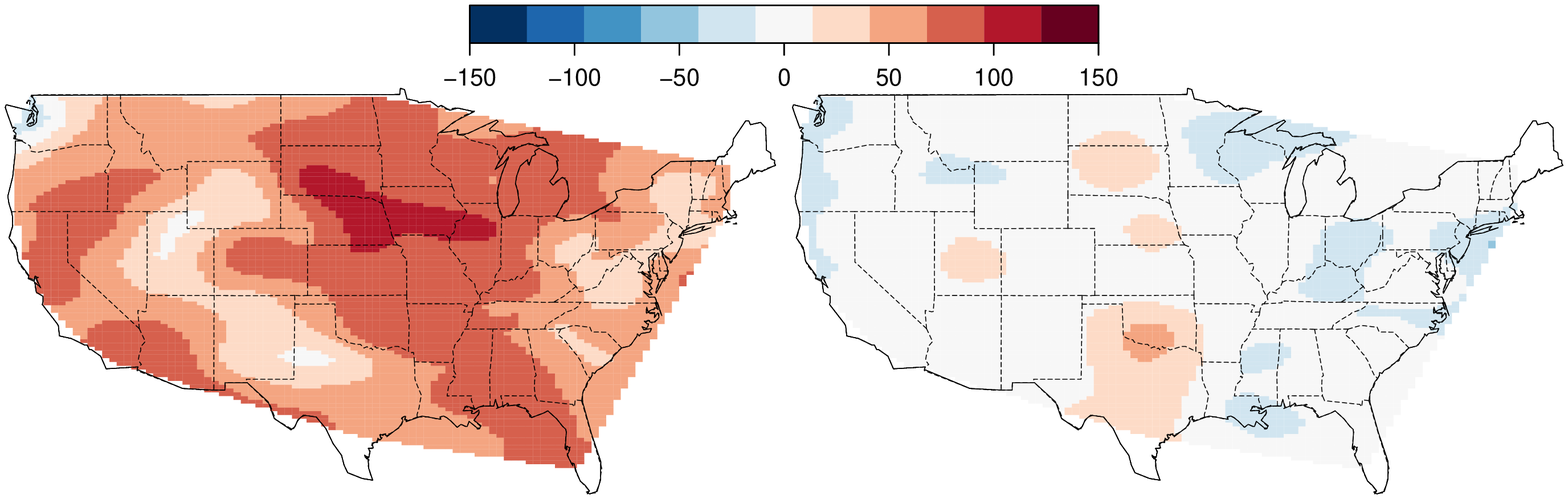}\\\vspace{-2.5em}
  \includegraphics[width=\textwidth]{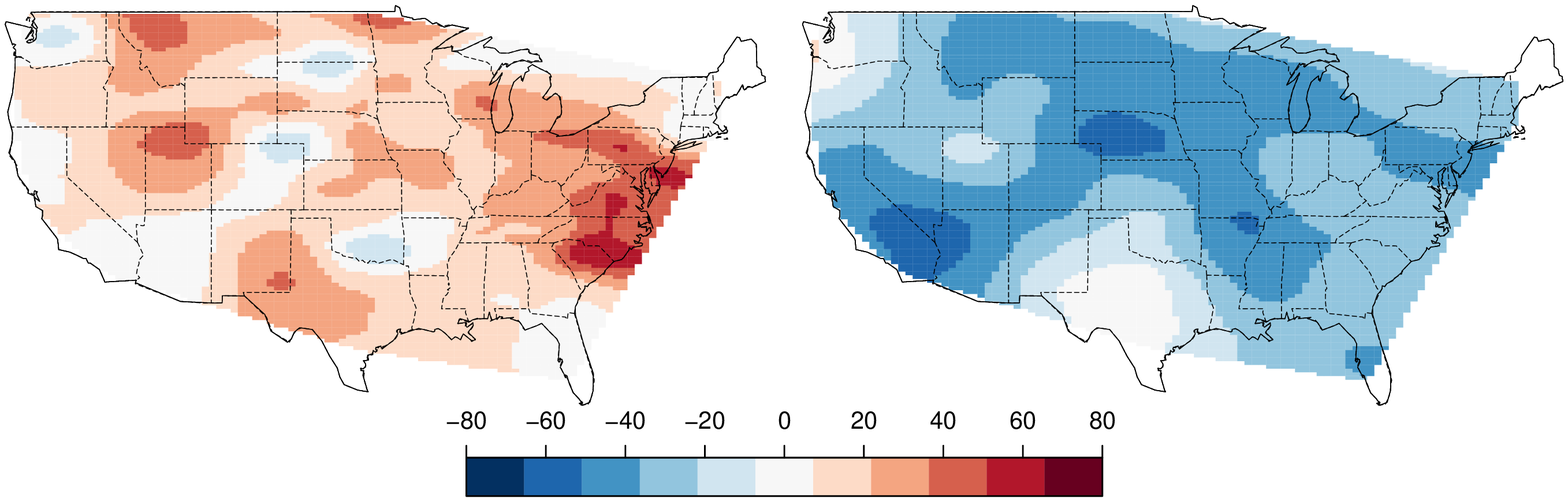}
  \caption{Spatial distribution of the estimated concurrence cell
    areas anomalies in squared degree, i.e., around $1000~\mbox{km}^2$
    for winter minima (top) and summer maxima (bottom). The data were
    stratified into three classes: El Niño, La Niña and the base class
    ``La Nada''. The left panels show the anomalies for El Niño, the
    right ones La Niña.}
\label{fig:cellAreaElNinho_vs_LaNinha} 
\end{figure} 

Finally, we consider another cut of the data by stratifying according
to an important climate phenomenon known as the El Niño Southern
Oscillation (ENSO).  Positive ENSO (El Niño) refers to the event of a
warm-up of the surface water in the central and east-central
equatorial Pacific ocean. It is well known that years with high ENSO
have a general warming effect in North America during the winter
season.  The opposite effect of negative ENSO (La Niña) is
characterized by a cool-down in the same area of the Pacific and it
generally leads to unusually cold winters in the northwestern part of
the US, northern California and the north-central
states~\cite{LaNinha_explanation}.
Figure~\ref{fig:cellAreaElNinho_vs_LaNinha} shows estimates of
concurrence cell areas anomalies for winter minima and summer
maxima. These anomalies were defined as pointwise deviations from the
expected cell area obtained from La Nada seasons, i.e., neither El
Niño nor La Niña seasons. We can see that La Niña does not seem to
have an impact on the spatial coverage of winter minima but that El
Niño seems to induce more massive cold extremes over the whole
USA\@. For the summer season, La Niña seems to reduce the spatial
extent of heat waves over the whole USA while El Niño has a less
pronounced impact---although it generally yields larger spatial
coverages especially along the East coast.

\section{Discussion}

In this paper we introduced a new framework for the analysis of
dependence of extremes: the \emph{extremal/sample concurrence
  probability}. This tool plays a similar role to that of the extremal
coefficient but has the benefit, as a probability, of being more
interpretable and intuitive. Theoretical properties and closed forms
of these concurrence probabilities have been established and several
estimators have been proposed. A simulation study has shown that the
proposed estimators work well in practice and that they give a new
insight about the dependence of extremes, such as the spatial
distribution of the expected concurrence cell area of extreme
temperature in the continental US\@.

\section*{Acknowledgements}

M. Ribatet was partly funded by the MIRACCLE-GICC and McSim ANR
projects. The authors gratefully acknowledge the help of Prof. Paul
H. Whitfield with the interpretation of the results on concurrence for
temperature extremes and also for the suggestion to stratify by El
Niño/La Niña effect.

\appendix{}

\section{Proofs}
\label{sec:proofs}

\subsection{Proof of the continuity of $\Pi$ in Theorem~\ref{thm:conv}}
\label{sec:proof-continuity-theta}

Consider a sequence $\Psi_n \to \Psi$ in
$\mathscr{M}_p([0, \infty]^k \setminus \{0\})$ and let
$\Psi = \{\psi_i\colon i \geq 1\}$ and
$\Psi_n = \{\psi_{n,i}\colon i \geq 1\}$. As compact subsets of
$[0, \infty]^k \setminus \{0\}$ are bounded away from $0$, we can
choose $\varepsilon > 0$ such that
\begin{equation*}
  \begin{cases}
    \max_{i \geq 1} \psi_i(s_j) > \varepsilon, & j = 1, \ldots, k\\
    \psi_i(s_j) \neq \varepsilon, & j = 1, \ldots, k,\  i \geq 1.
  \end{cases}
\end{equation*}
Then $K_\varepsilon=[0, \infty]^k \setminus [0, \varepsilon)^k$ is a
compact set, $\Psi \cap K_\varepsilon $ has finitely many points
$\psi_1, \ldots,\psi_{N_\varepsilon}$ and no point of $\Psi$ lies on
the boundary $\partial K_\varepsilon$.  The convergence
$\Psi_n \to \Psi$ entails that for $n$ large enough,
$\Psi_n \cap K_\varepsilon$ has the same number of points
$\psi_{n,1}, \ldots, \psi_{n,N_\varepsilon}$ that can be reordered in
such a way that $\psi_{n,i} \to \psi_i$ as $n\to \infty$,
$i = 1, \ldots, N_\varepsilon$.

Assume that the maxima
$\max_{1 \leq i \leq N_\varepsilon} \psi_i(s_j)$, $j=1, \ldots, k$,
are uniquely attained.  Then the hitting scenario $\Pi(\Psi)$ is well
defined and depends only on
$\{\psi_{i}\colon i = 1, \ldots, N_\varepsilon\}$. By the convergence
$\psi_{n,i}\to\psi$, for large $n$ the maxima
$\max_{1\leq i\leq N_\varepsilon} \psi_{n,i}(s_j)$ are uniquely
attained so that the hitting scenario $\Pi(\Psi_n)$ is well defined
and depends only on
$\{\psi_{n,i}\colon i = 1, \ldots N_\varepsilon\}$.  It is not
difficult to see (although tedious to write formally) that the
convergence
$\{\psi_{n,i}\colon i = 1, \ldots, N_\varepsilon\} \to
\{\psi_{i}\colon i = 1, \ldots, N_\varepsilon\}$
implies $\Pi(\Psi_n)=\Pi(\Psi)$ for large $n$. This proves the
announced continuity for the mapping $\Pi$.

\subsection{Proof of Proposition~\ref{prop:bias-estimate}}
\label{sec:proof-bias-estimate}
 
 We shall prove below the following formula, which may be of
 independent interest.
 
 \begin{lem}\label{lem:pm-p}
   In the context of Proposition~\ref{prop:bias-estimate}, we have
   \begin{equation}\label{eq:bias-estimate}
     p_m=p+\sum_{\ell=2}^k \Pr(|\pi|=\ell)m^{1-\ell}
   \end{equation}
   where $\pi$ is the extremal hitting scenario and $|\pi|$ its number
   of components.
\end{lem}

Proposition~\ref{prop:bias-estimate} follows directly from
\eqref{eq:bias-estimate}.  Indeed, the monotonicity of $p_m$ is
immediate and since $p = \Pr(|\pi| = 1)$, we have
$$
0\leq p_m - p = \sum_{\ell=2}^k \frac{\Pr(|\pi|=\ell)}{m^{\ell-1}}
\leq \frac{1}{m} \sum_{\ell=2}^k \Pr(|\pi|=\ell) = \frac{(1-p)}{m}.
$$
If $p<1$, then at least one of the probabilities $\Pr(|\pi| = \ell)$,
$\ell = 2,\ldots,k$ is non-zero and by~\eqref{eq:bias-estimate} the
asymptotic equivalence $(p_m-p)\sim c_r/m^r$, $m\to\infty$ holds where
$r\in\{1,\ldots,k-1\}$ is the smallest integer, such that
$c_r= \Pr(|\pi| = r+1)>0$.

\begin{proof}[Proof of Lemma~\ref{lem:pm-p}]
Since $Z_1, \ldots, Z_m$ are independent with the same distribution as
$\eta$, we can suppose from~\eqref{eq:spectralCharacterizationv2} that
\begin{equation*}
  Z_i(s)=\max_{\phi\in \Phi_i} \phi(s),\qquad s\in\mathcal{X},
\end{equation*}
with $\Phi_1,\ldots,\Phi_m$ independent copies of $\Phi$. By
max-stability, $\bar \eta=m^{-1}\max_{1\leq i\leq m}Z_i$ has the same
distribution as $\eta$ and $\bar\Phi=\cup_{1\leq i\leq
  m}\{m^{-1}\phi,\phi\in \Phi_i\}$ has the same distribution as
$\Phi$. We consider the following events
\begin{align*}
  A&= \{\text{sample concurrence occurs for $Z_1,\ldots,Z_m$} \}\\
  B&=\{\text{extremal concurrence occurs for $\bar \eta$} \}
\end{align*}
Clearly, $\Pr(A)=p_m$ and $\Pr(A)=\Pr(A\cap B)+\Pr(A\cap B^c)$ with
$B^c$ the complementary set of $B$. We analyze the two terms
separately.

Observe first that if extremal concurrence occurs then we also have
sample concurrence. Indeed, if $B$ occurs, then one function of
$\bar\Phi$ dominates all the others at $(s_1,\ldots,s_k)$.  This
function is of the form $m^{-1}\phi$ with $\phi\in\Phi_i$, for some
$i =1,\ldots,m$, showing that $Z_i$ dominates $Z_1,\ldots,Z_m$, i.e.,\
we have sample concurrence. Hence $B\subset A$ and
$\Pr(A\cap B)=\Pr(B)=p$. We now consider the second term and the event
$A\cap B^c$, i.e., sample concurrence occurs in $Z_1,\ldots,Z_m$ but
not extremal concurrence for $\bar \eta$. Let $\bar\pi$ be the hitting
scenario of $\bar\eta$. We know that $B^c$ is equivalent to
$|\bar\pi|\geq 2$, i.e.\ the maximum at locations $(s_1,\ldots,s_k)$
is attained by at least two functions in $\bar\Phi$.  These functions
are of the form $m^{-1}\phi_j$, $1\leq j\leq \ell$, with
$\ell=|\bar \pi|$ and $\phi_j\in \Phi_{i_j}$ for some
$1\leq i_j\leq m$. If $A$ is also realized, i.e., some $Z_i$ dominates
$Z_1,\ldots,Z_m$, then we must have $i_1=\cdots=i_\ell=i$. Note,
however, that since the point processes $\Phi_i$, $i=1, \ldots, m$ are
independent and identically distributed, any given function
$m^{-1}\phi\in \bar\Phi \equiv \cup_{i=1}^m m^{-1}\Phi_i$,
independently from the others, has equal chance of coming from any one
of the $m$ point processes $m^{-1}\Phi_i$, $i=1,\ldots,m$.  Therefore,
the probability that all $\ell$ functions contributing to the maximum
at sites $(s_1, \ldots, s_k)$ are assigned to component $i$ is
$m^{-\ell}$. Since there are $m$ possible choices for the index $i$,
we deduce
\begin{equation*}
  \Pr(A\cap B^c)=\sum_{\ell=2}^k \Pr(|\bar \pi|=\ell) m^{1-\ell}.
\end{equation*}
Equation~\eqref{eq:bias-estimate} follows.
\end{proof}

\subsection{Proofs of Theorems~\ref{thm:CLT-1} and~\ref{thm:CLT-2}}
\label{sec:proofs-CLT}

In the context of these two theorems, we have
\begin{equation}\label{e:pm-p}
  p_m - p \sim c_r m^{-r}, \qquad m\to\infty,
\end{equation}
for some $r \in \{1,\ldots,k-1\}$ and
$c_r>0$---cf. Proposition~\ref{prop:bias-estimate}.  Let
$S_n = [n/m]\hat p_m = \sum_{i=1}^{[n/m]} \xi_{i,n}$, where
$\xi_{i,n}$ are iid Bernoulli$(p_m)$.

\begin{proof}[Proof of Theorem~\ref{thm:CLT-1}] 
Let $B_n = {\rm Var}(S_n) = [n/m] p_m(1-p_m)$ and introduce the cumulative
distribution function
\begin{equation*}
  F_n(x)  = \Pr \left[ B_n^{-1/2} \{S_n - \E( S_n) \}  \leq x \right]
  \equiv \Pr\left\{ \sqrt{\frac{[n/m]}{p_m(1-p_m)}} (\hat p_m - p_m)
    \le x\right\}, \qquad x \in \mathbb R.
\end{equation*}
The Berry--Essen theorem (see e.g.\ Theorem V.2.3 in
\cite{petrov:1975}) implies that
\begin{equation}\label{e:BE-bound}
  \sup_{x\in \mathbb R} | F_n(x) - \Phi(x)| \le A L_n,
\end{equation}
where $A$ is an absolute constant, $\Phi(x)$ denotes the standard
Normal cumulative distribution function and
\begin{equation*}
  L_n = B_n^{-3/2}\sum_{i=1}^{[n/m]} \E \left( |\xi_{i,n} - p_m|^3 \right).
\end{equation*}
Using that $\E (|\xi_{i,n} - p_m|^{3}) \le p_m(1-p_m)$ and
straightforward algebra, we obtain
\begin{equation}\label{e:BE-bound-L}
L_n \le \left[  (n/m) p_m(1-p_m)\right]^{-1/2}.
\end{equation}

Since $0<p<1$, we have $p_m(1-p_m)\sim p(1-p) >0$ as $n\to\infty$ and,
for an arbitrary choice of $m= m(n) = o(n)$ as $n\to\infty$, we have
$L_n\to 0$. This, in view of~\eqref{e:BE-bound} yields
\begin{equation}
\label{e:BE-bound-conv}
  \sqrt{n/m} (\hat p_m - p_m) \longrightarrow N\{0,p(1-p)\}.
\end{equation}
By using~\eqref{e:pm-p} and Slutsky's theorem, we obtain the final result.
\end{proof}

\begin{proof}[Proof of Theorem~\ref{thm:CLT-2}]
  In case $\lambda<\infty$, since $p=0$, from~\eqref{e:pm-p}, we have
  $p_m\sim c_r/m^r,\ n\to\infty$, and hence
  \begin{equation}\label{e:CLT-2}
    [n/m] p_m \sim c_r n/m^{(r+1)} \longrightarrow
    c_r/\lambda^{(r+1)}, \qquad n\to\infty. 
  \end{equation}
  Thus, the standard Poisson convergence for the Binomial$(p_m,[n/m])$
  random variables $S_n = [n/m] \hat p_m$ yields the result.

  When $\lambda = \infty$, by~\eqref{e:CLT-2}, we have
  $x_n = [n/m] p_m\to 0$ as $n\to\infty$. Therefore,
  \begin{equation*}
    \Pr(\hat p_m = 0) = (1-p_m)^{[n/m]} = \left(1-
      \frac{x_n}{[n/m]}\right)^{[n/m]} \sim \exp(-x_n) \longrightarrow
    1, \qquad n \to \infty,
  \end{equation*}
  which completes the proof.
\end{proof}

\bibliography{biblio}
\bibliographystyle{apalike}
\end{document}